\documentclass{article}
\usepackage{epsfig}

\usepackage{epsfig}
\usepackage{graphicx, amsmath, amssymb, amsthm, amscd, fullpage,verbatim,hyperref,mathrsfs,latexsym}

\def\qed{\ifvmode\mbox{ }\else\unskip\fi\hskip 1em plus 10fill$\Box$}

\input{epsf}

\newtheorem{prelem}{{\bf Theorem}}

\newtheorem{theorem}{Theorem}[section]
\newtheorem{corollary}[theorem]{Corollary}
\newtheorem{definition}[theorem]{Definition}

\newtheorem{lemma}[theorem]{Lemma}

\newtheorem{example}[theorem]{Example}

\newtheorem{remarka}[theorem]{Remark}
\newenvironment{remark}{\begin{remarka}\rm}{\end{remarka}}

\newtheorem{problema}[theorem]{Problem}

\def\deg {{\rm deg}}
\def\L {{\Lambda}}
\def\R {{\mathbb R}}
\def\Z {{\mathbb Z}}
\def\N {{\mathbb N}}
\def\l {{\ell}}
\def\onev {\vec{\bf 1}}
\def\zerov {\vec{\bf 0}}

\title{\bf Chip-Firing and Riemann-Roch Theory for Directed Graphs}
\author{
{Arash Asadi\footnote{{Email: \href{mailto:aasadi@math.gatech.edu} {\texttt{aasadi@math.gatech.edu}}}}} \,\,\,\,\,\,\,\,\,\,\,\,  {Spencer Backman\footnote{{Email:
   \href{mailto:sbackman@math.gatech.edu} {\texttt{sbackman@math.gatech.edu}}}}}\\
   \\
{\small\it School of Mathematics}\\
{\small\it Georgia Institute of Technology} \\
{\small\it Atlanta, GA, USA} \\}
\begin{document}
\maketitle

\begin{abstract}
We investigate Riemann-Roch theory for directed graphs. The Riemann-Roch criteria of Amini and Manjunath is generalized to all integer lattices orthogonal to some positive vector.  Using generalized notions of a $v_0$-reduced divisor and Dhar's algorithm we investigate two chip-firing games coming from the rows and columns of the Laplacian of a strongly connected directed graph.  We discuss how the ``column" chip-firing game is related to directed $\vec{G}$-parking functions and the ``row" chip-firing game is related to the sandpile model.   We conclude with a discussion of arithmetical graphs, which after a simple transformation may be viewed as a special class of directed graphs which will always have the Riemann-Roch property for the column chip-firing game. Examples of arithmetical graphs are provided which demonstrate that either, both, or neither of the two Riemann-Roch conditions may be satisfied for the row chip-firing game.
\end{abstract}

\tableofcontents

\section{Introduction \label{intro_sec}}
We use standard graph theoretic and discrete geometric notation and terminology, which may be found in~\cite{Bol98, Die05} and ~\cite{Mat02, Sch98} respectively. All graphs in this paper are finite and have no loops although they may be directed or have multiple edges (multi-graphs). We refer the reader to Section~\ref{notedef_sec} for some basic notation and definitions.

Let $R$ be a positive $n+1$ dimensional vector and $\L_R=\{D \in \Z^{n+1}: D \cdot R=0\}$.  Fix $\L$, a full-dimensional sub-lattice of $\L_R$. As noted in~\ref{notedef_sec}, we refer to an element $D \in \Z^{n+1}$ as a divisor. We say divisors $D, D' \in \Z^{n+1}$ are {\it equivalent}, denoted by $D \sim D'$, if and only if $D-D'\in \L$. We say a divisor $E \in \Bbb Z^{n+1}$ is {\it effective} if $E \geq {\zerov}$.
For any divisor $D \in \Z^{n+1}$, the {\it linear system} associated to $D$ is the set $|D|$ of all effective divisors which are equivalent to $D$, i.e.,
$|D|=\{E \in \Z^{n+1}: E \geq {\zerov}, E \sim D\}$ and the {\it degree} of $D$, written $deg_R(D)$, is given by $D\cdot R$.

\begin{definition}
\label{rank_function_def}
For any divisor $D \in \Z^{n+1}$, define the rank of $D$, denoted by $r(D)$, as follows: $$r(D)=\min\{\deg_R(E): |D-E|=\emptyset, E \geq {\zerov}\}-1.$$
\end{definition}

Baker and Norine~\cite{BN07} developed a graph theoretic analogue of the Riemann-Roch formula, originally by studying a certain unrestricted chip-firing game on graphs.  Geometrically their result states that for the lattice $\L_G$ spanned by the rows of the Laplacian of a finite undirected graph $G$, there exists a {\it canonical} divisor $K \in \Z^{n+1}$ whose $i$-th entry is $\deg(v_i)-2$, of degree $2g-2$ (where $g=|E(G)|-|V(G)|+1$) such that for any divisor $D \in \Z^{n+1}$,
\begin{equation}
\label{RR_Baker_Norine}
r(D)-r(K-D)=\deg_{\onev}(D)+1-g.
\end{equation}
Many of their results have since been generalized to a variety of objects including tropical curves, metric graphs and edge weighted graphs~\cite{GK08, HKN07, Luo08, MZ07}. Recently Amini and Manjunath~\cite{AM09} showed that by viewing a the chip-firing game of Baker and Norine geometrically as a walk through the lattice spanned by its Laplacian, a pair of necessary and sufficient Riemann-Roch conditions, equivalent to those of Baker and Norine, could be generalized to all sub-lattices of the lattice $\L_{\onev}$.  They refer to these conditions as {\it uniformity} and {\it reflection invariance}.

In Section~\ref{lat_sec}, Theorem~\ref{RR_formula_equiv_U_RI_thm} shows that the criteria of Amini and Manjunath~\cite{AM09} naturally extends to any full-dimensional sublattice of $\L_R$. Lorenzini~\cite{Lor09} gives an alternate Riemann-Roch criteria for such lattices.  Our approach differs from his in that we first give a specific rank function (Definition~\ref{rank_function_def}) and use this to define a pair of necessary and sufficient conditions for a lattice $\L$ to have the Riemann-Roch property. Lorenzini~\cite{Lor09} instead says that such a lattice has the Riemann-Roch property if there exists a {\it suitable} rank function (\S 2.1 in~\cite{Lor09}), i.e., one which would allow for a Riemann-Roch formula (\ref{RR_Baker_Norine}) satisfying certain desirable properties.  We conclude section 2 with Theorem~\ref{thm:RR_R_one} showing that a full-dimensional lattice $\L \subseteq \L_R$ has the Riemann-Roch property if and only if $\mathcal{R}\L \subseteq \L_{\onev}$ does, where $ \mathcal{R}=diag(r_0, \dots, r_{n})$.  This result is later employed when studying the column chip-firing game and when discussing the relationship of chip-firing on arithmetical graphs to the row chip-firing game on associated direct graphs.

Various chip-firing games on graphs have been studied in ~\cite{Big99, BL92, BLS91, HLMPPW08, GR01, Mer05, Mer97, Spe93, Tar88, Heu88}. Baker and Norine~\cite{BN07} introduced an unrestricted chip-firing game on undirected graphs to prove their Riemann-Roch formula.  Their game is as follows: begin with a graph and an integer number of ``chips'' at each vertex.  A vertex either {\it borrows} a chip along each of its edges from its neighbors or it {\it fires}, sending a chip along each of its edges to its neighbors.  The objective of the game is to bring all of the vertices out of debt.  In Section~\ref{chip_sec}, we investigate two separate generalizations of the unrestricted chip-firing game on undirected graphs to directed graphs.  To understand the two different generalizations of this game to directed graphs we should study how this game relates to the graph Laplacian.  The question of whether a configuration $D$, also called a divisor, can be brought out of debt by some sequence of firings and borrowings is the equivalent to the question of whether $|D|\neq \emptyset$, i.e., $r(D)\geq 0$ for the lattice $\L_G$.  This is because a sequence of chip-firings corresponds to translation by a lattice point in $\L_G$.  Let $\vec{G}$ be a directed graph whose adjacency matrix $\vec{A}$ with $i$, $j$th entry $\vec{A}_{i,j}$ is the number of edges directed from $i$ to $j$.  Let $\vec{\mathcal{D}}=diag(\deg^+(v_i), \dots, \deg^+(v_n))$ where $\deg^+(v)$ denotes the number edges leaving vertex $v \in V(\vec{G})$. We call the matrix $\vec{Q}=\vec{\mathcal{D}}-\vec{A}$ the {\it Laplacian matrix} of the directed graph $\vec{G}$. Note that this directed Laplacian is symmetric if and only if it is the Laplacian of a graph with bidirected edges, i.e.,  an undirected graph.  We investigate $r(D)$ and the Riemann-Roch formula for the lattice spanned by the rows of $\vec{Q}$ and the lattice spanned by the columns of $\vec{Q}$.  For both of these lattices, it is equivalent to study certain chip-firing games on $\vec{G}$.  We note that throughout the paper the directed graphs being studied are constrained to be strongly connected, i.e., for any two vertices $i, j \in V(\vec{G})$, there exists a directed path from $i$ to $j$.

Studying the lattice spanned by the rows of the directed Laplacian is equivalent to studying the row chip-firing game in which if a vertex fires, it will send a chip along each of its outgoing edges.  In [2], an important object, called a $v_0$-reduced divisor, is introduced.  Essentially this is a configuration, where every vertex is out of debt with the possible exception of $v_0$ and there is no way of "pushing" any money towards $v_0$.  We generalized this notion of a $v_0$ reduced divisor to the row chip-firing game on a strongly connected directed graph in Section~\ref{reduceddiv_sec}. In Section~\ref{subsec:Dhar_alg}, we generalize Dhar's Algorithm, which Baker and Norine used implicitly in~\cite{BN07}.  Dhar's algorithm allows one to check whether a divisor whose entries are nonnegative for all vertices other than $v_0$ is $v_0$-reduced and gives, when the divisor is reduced, all of the equivalent $v_0$-reduced divisors (for the case of directed graphs, a $v_0$-reduced divisor is no longer in general unique).  When the divisor is found to not be $v_0$-reduced, a firing is obtained, which will bring it ``closer'' to some $v_0$-reduced divisor.  In section 4 we present examples which show that lattice spanned by the rows of $\vec{Q}$ may or may not have the Riemann-Roch formula.

We say a directed graph has the strong Riemann-Roch property for directed graphs if it has the Riemann-Roch property and it has a canonical vector $K$ whose $i$th entry $K(v_i)$ is $\deg^+(v_i)-2$.
We then mention a connection between the sandpile model and the Riemann-Roch property for the row chip-firing game in Section~\ref{subsec:sandpile}.  The directed sandpile model is a constrained version of the row chip-firing game where we restrict our attention to effective divisors.  We fire vertices only when they have at least as many chips as their outdegree (so that the divisor remains effective).  While many authors require a global sink and ignore the number of chips at this vertex.  Because we are studying strongly connected digraphs it is sufficient for our discussion to simply require that a specified vertex $v_0$ not fire.  A divisor $D$ is stable if stable if no vertices may fire and a stable divisor $D$ is recurrent if for every other divisor there exists a way of adding chips to vertices after which the divisor will stabilized to $D$.  We show that for a directed graph $\vec{G}$, the lattice $\L_{\vec{G}}$ has the strong Riemann-Roch property for directed graphs if and only if for every $v_0$-recurrent sandpile configuration $D$, which is minimal with respect to dominance away from $v_0$, there exists $D'=D-ke_0, k \in \Z_{\geq0}$, which is a continuous extreme divisor.  The notion of a continuous extreme divisor is introduced in section 2 and is equivalent to saying that there exist $E_i \in \Z_{\geq 0}$ for $0 \leq i \leq n$ such that $E_i(v_i)=0$ and $E_i(v_j)>0$ for $i \neq j$ and $D' \sim E_i$.  We note that $v_0$-reduced divisors, their connection to $v_0$-recurrent sandpile configurations and the generalized Dhar's algorithm were independently discovered by Speer~\cite{Spe93} although he was not aware of the connection with Riemann-Roch theory.

Studying the the lattice spanned by the columns is equivalent to studying the column chip-firing game in which if a vertex borrows, it sends a chip along each of its incoming edges and loses a number of chips equal to its outdegree.  The number of chips is not conserved, but if we restrict our attention to strongly connected digraphs then we find that there exists a canonical set of currencies, which are integer multiples of some universal currency, with exchange rates so that the game is conservative.  In Section~\ref{subsec:G-parking}, we explain that the $v_0$-reduced divisors for this game are precisely the directed $G$-parking functions studied in ~\cite{CP05}.  We show that when studying the column chip-firing game on a strongly connected graph, it is equivalent to study the row chip-firing game on an associated Eulerian directed graph, that is, a directed graph for which each vertex has the same number of outgoing and incoming edges.  We also mention how Dhar's algorithm can be run on a divisor in the column chip-firing game without any serious revision.

We then consider the case of {\it arithmetical graphs} in Section~\ref{AGrpahs}.  An arithmetical graph is an undirected multigraph along with a vector $R \in \N^{n+1}$, with $R= (r_0, ..., r_n)$, where $r_i$ is the weight of vertex $v_i$ subject to the constraint that the sum of the weights of the vertices adjacent to $v_i$ (counting with multiplicity equal the number of edges shared with $v_i$) you obtain $\delta_ir_i$ for some $\delta_i \in \N$.  We define the Laplacian of an arithmetical graph to be the same as for a standard multigraph, but with the $i$th entry along the diagonal equal to $\delta_i$ instead of the degree of $v_i$.  Lorenzini~\cite{Lor89} introduced arithmetical graphs as a way of studying the intersection matrices of degenerating curves, which encode some of the discrete data associated with the degeneration.  In this paper our interest in arithmetical graphs is derived from the fact that they form a class of vertex weighted graphs whose Laplacian spans an $n$-dimensional sub-lattice of $\L_R$.
Indeed, Chung and Langlands~\cite{CL96} introduced a Laplacian matrix for a graph with weights on its vertices, and noted in~\cite{Lor09} that if for all $0 \leq i \leq n$ the weight of the vertex $v_i$ is the square of the positive integer $r_i$, the Laplacian matrix introduced in~\cite{CL96} is the same as the one defined above.  The chip-firing game of Baker and Norine extends to arithmetical graphs by assigning to each vertex its own currency, interpreting each vertex's multiplicity as the integer exchange rate between this vertex's currency and the universal chip currency.  This is very similar to the notion of currencies employed when studying the column chip-firing game.  In doing so we are able to give a combinatorial interpretation of the geometric definitions and statements of Section~\ref{lat_sec} for arithmetical graphs.

We may obtain from an arithmetical graph $(G,R)$ with Laplacian $Q$, the Laplacian $\vec{Q}=Q\mathcal{R}$ (where $ \mathcal{R}=diag(r_0, \dots, r_{n})$) of a closely related directed graph.  In this way we may view arithmetical graphs as a special type of directed graph, particularly since this coordinatewise scaling reduces the chip-firing game for arithmetical graphs to the row chip-firing game for directed graphs and preserves the Riemann-Roch property by Theorem~\ref{thm:RR_R_one}.  In Theorem~\ref{asdf} we show that the all of the associated directed graphs have the Riemman-Roch property for the column chip-firing game.

Given an arithmetical graph $(G,R)$ we define $g_0$ by the formula $2g_0-2 = \sum_{i=0}^n r_i(\delta_i-2)$.  See~\cite{Lor89} for a simple that $g_0$ is integral and note that $g_0$ is $g$ for a graph
 $(G,\onev)$.  As an application of the tools developed in section 3 we give a combinatorial proof of Proposition 4.2 from~\cite{Lor09}, which states that $g_{max} \leq g_0$ and if $g_{min}=g_{max}= g_0$ then $(G,R)$ has the Riemann-Roch property (and in particular the associated directed graph has the Riemann-Roch property).  The first half of this statement, in the language of chip-firing, says that if there are $g_0$ chips present in an arithmetical graph then there exists a winning strategy thus generalizing the result of Baker and Norine for arithmetical graphs.  The original proof of this result due to Lorenzini was algebro-geometric in nature, employing Riemann-Roch formula for curves.

We conclude with a discussion of some examples of arithmetical graphs, which demonstrate that either, both, or neither of the two Riemann-Roch conditions may be satisfied for an arithmetical graph.
\subsection{Basic Notations and Definitions \label{notedef_sec}}
For any two vectors $x,y \in \Bbb R^{n+1}$, let $x \cdot y$ denote the inner product of $x$ and $y$.
For any $x=(x_0, \dots, x_n)^T \in \Bbb R^{n+1}$, define $x^+=(x^+_0, \dots, x^+_n)^T \in \R_+^{n+1}$ and  $x^-=(x^-_0, \dots, x^-_n)^T \in \R_-^{n+1}$ to be the {\it positive part} and {\it negative part} of $x$ where $x=x^++x^-$ and $x_i^+x_i^-=0$, for all $0 \leq i \leq n$. Define $\deg_R(x)=R \cdot D$ and call it the {\it degree} of $x$. We denote $\deg_R(x^+)$ by $\deg^+_R(x)$ and we call it the {\it degree plus} of $x$.\\
Assume ${\zerov}$ and ${\onev}$ are the vectors in $\R^{n+1}$  all of whose coordinates are $0$ or $1$, respectively. For any $x=(x_0, \dots, x_n)^T \in \R^{n+1}$, we say $x \geq {\zerov}$ ($x > {\zerov}$) if and only if for all $0 \leq i \leq n$, $x_i \geq 0$ ($x_i > 0$). We define a {\it partial order} in $\R^{n+1}$ as follows: for any $x,y \in \R^{n+1}$, we say $x \geq y$ ($x > y$) if and only if $x-y \geq {\zerov}$ ($x-y > {\zerov}$).
For any vector $x \in \Bbb R^{n+1}$, define $C^+(x)=\{y \in \Bbb R^{n+1}: y \geq x\}$ and $C^-(x)=\{y \in \Bbb R^{n+1}: x \geq y\}$. We denote the standard basis for $\R^{n+1}$ by $\{e_0, \dots, e_n\}$. Suppose that $R \in \N^{n+1}$ is a vector, and define $H_R=\{x \in \Bbb R^{n+1}: R \cdot x=0\}$. Let  $\L_R=H_R \cap \Z^{n+1}$ be the integer lattice in the hyperplane $H_R$ where $R \in \N^{n+1}$. Let $\|\cdot\|$ denote the $\ell^2$-norm, i.e., $\|x\|= \sqrt{x \cdot x}$, for all $x \in \R^{n+1}$.

Let $G$ be graph and let $\{v_0, \dots, v_n\}$ be an ordering of vertices of $G$. Let $Div(G)$ be the free Abelian group on the set of vertices of $G$. By analogy with the Riemann surface case as noted also in~\cite{BN07}, we refer to elements of $Div(G)$ as {\it divisors} on $G$. In the case that the graph $G$ is implied by context, we simply refer to elements of $Div(G)$ as divisors. Because there is a fixed ordering on vertices of $G$, we think of an element $\alpha \in Div(G)$, which is a formal integer linear combinations of vertices of $G$, as a vector $D=(d_0, \dots, d_n) \in \Z^{n+1}$ where  $d_i$ is the coefficient of $v_i$ in $\alpha$ for all $0 \leq i \leq n$. We denote to the $i$th coordinate of $D$ by $D(v_i)$, for all $0 \leq i \leq n$.  We refer to both vectors in $\Z^{n+1}$ and elements of $Div(G)$ as divisors.
\section{Riemann-Roch Theory for Sub-lattices of $\L_R$ \label{lat_sec}}
\subsection{Preliminaries}

We remark that many of the proofs and statements presented in this section are similar to the ones which appeared in Amini and Manjunath~\cite{AM09}'s work.  Essentially, what is being demonstrated is that if one replaces each statement about lattices orthogonal to the all one's vector with the same statement for lattices orthogonal to some fixed positive vector, the proofs will go through without much extra effort.  This in itself is not a very strong observation, but it is necessary for proving Theorem~\ref{RR_formula_equiv_U_RI_thm} and Theorem~\ref{thm:RR_R_one}, which are used several times in the proceeding sections so, for the sake of completeness, we have decided to provide all of the necessary lemmas with proofs.

Throughout this section, $R$ will denote a vector in $\N^{n+1}$.
\begin{definition}
\label{sigma_region_def}
 Let $\L \subseteq \L_R$ be a sub-lattice of rank $n$. Define
$$\Sigma(\L)=\{D \in \Bbb Z^{n+1}: D \not \geq p {\hbox{ for all }} p \in \L\},$$
$$\Sigma_{\Bbb R}(\L)=\{x \in \Bbb R^{n+1}: x \not \geq p {\hbox{ for all }} p \in \L\}.$$
\end{definition}
Note that the set $\Sigma(\L)$ defined in Definition~\ref{sigma_region_def} is the negative of the {\it Sigma region} set defined by Amini and Manjunath~\cite{AM09}.
We denote by $\overline{\Sigma}_{\Bbb R}(\L)$ the topological closure of the set $\Sigma_{\Bbb R}$ in $\Bbb R^{n+1}$. Let $B(x,r)=\{y \in \R^{n+1}: \|y-x\|\leq r\}$ denote the ball of radius $r$ with center at $x$. For any set $S \subset \Bbb R^{n+1}$, let $int(S)$ denote the {\it relative interior} of $S$.

\begin{lemma}
\label{sigma_closure}
If $\L \subseteq \L_R$ is a sub-lattice of rank $n$, then $$\overline{\Sigma}_{\Bbb R}(\L)=\{x \in \Bbb R^{n+1}: x \not > p, {\hbox{ for all }} p \in \L\}.$$
\end{lemma}
\begin{proof}
{
Suppose $x \in \Bbb R^{n+1}$ such that $x>p$ for some $p \in \L$. Thus there exists $\delta>0$ such that for all $y \in B(x,\delta)$, $y>p$. Thus $x \not \in \overline{\Sigma}_{\Bbb R}(\L)$. Now, suppose $x \not \in \overline{\Sigma}_{\Bbb R}(\L)$. Then there exists $\delta>0$ and $p \in \L$ such that $x-{\delta \over 2}{\onev} \geq p$. Hence $x>p$, and this completes the proof of the lemma.
}
\end{proof}
\begin{lemma}
\label{sigma_closure_sigma_lem}
If $D \in \Bbb Z^{n+1}$ then $D \in \Sigma(\L)$ if and only if $D+{\onev} \in \overline{\Sigma}_{\Bbb R}(\L)$.
\end{lemma}
\begin{proof}
{
If $D \not \in \Sigma(\L)$, then there exists $p \in \L$ such that $D \geq p$. Hence $D+{\onev} > p$ and by Lemma~\ref{sigma_closure} $D+{\onev} \not \in \overline{\Sigma}_{\Bbb R}(\L)$. If $D+{\onev} \not \in \overline{\Sigma}_{\Bbb R}(\L)$ then Lemma~\ref{sigma_closure} implies that $D+{\onev} > p$ for some $p \in \L$. Since $D,p \in \Bbb Z^{n+1}$, it follows that $D \geq p$ and this implies that $D \not \in \Sigma(\L)$.
}
\end{proof}
Suppose $R=(r_0, \dots,r_{n}) \in \R^{n+1}_+$ and $x=(x_0, \dots,x_n) \in \R^{n+1}$. Define $\|x\|_R=\sum_{i=0}^n r_i|x_i|$. It is easy to see that $\| \cdot \|_R$ is a norm on $\Bbb R^n$. For any two points $x,y \in \Bbb R^{n+1}$, we define $dist_R(x,y)=\|x-y\|_R$.  One can consider $\| \cdot \|_{R}$ as a {\it weighted taxi-cab} distance.  For any set $S \subseteq \R^{n+1}$ and $p \in \R^{n+1}$, we define $dist_R(p,S)=\inf\{dist_R(p,x): x \in S\}$.
Observe that $r(D)=-1$ if $D$ is not equivalent to any effective divisor and $-1 \leq r(D) \leq \deg_R(D)$.
\begin{lemma}
\label{distance_rank_lem}
If $D \in \Bbb Z^{n+1}$ is a divisor then
\begin{itemize}
 \item[(i)] $r(D)=-1$ if and only if $D \in \Sigma(\L)$.
 \item[(ii)] $r(D)=dist_R(D,\Sigma(\L))-1=\min\{dist_R(D,p): p \in \Sigma(\L)\}-1$.	
\end{itemize}
\end{lemma}
\begin{proof}
{
 \begin{itemize}
  \item [(i)]For $D \in \Bbb Z^{n+1}$, $r(D)=-1$ if and only if for all $p \in \L$, $D-p \not \geq {\zerov}$ if and only if $D \in \Sigma(\L)$.
   \item [(ii)]Since $\Sigma(\L)$ is a closed set, $\inf\{dist_R(D,p): p \in \Sigma(\L)\}=\min\{dist_R(D,p): p \in \Sigma(\L)\}$.\\
\begin{eqnarray*}
 r(D) &=& \min\{\deg(E): |D-E|= \emptyset, E \geq {\zerov}\}-1 \\
    &=& \min\{\deg(E): r(D-E)=-1, E \geq {\zerov}\}-1 \\
   &=& \min\{\deg(E): D-E \in \Sigma(\L), E \geq {\zerov}\}-1 \\
   &=& \min\{\deg_R(D-p): D-p \geq {\zerov}, p \in \Sigma(\L)\}-1 \\
   &=& 	dist_R(D,\Sigma(\L))-1.
\end{eqnarray*}
Note that the last equality follows from the fact that if $p \in \Sigma(\L)$ and $(D-p)_i < 0$ for some $0 \leq i \leq n$ then $dist_R(D, p-e_i) \leq dist_R(D, p)$ and $p-e_i \in \Sigma(\L)$.  	
\end{itemize}
}
\end{proof}
\subsection{Extreme Points of $\Sigma(\L)$ and $\overline{\Sigma}_{\Bbb R}(\L)$}
Define $H_R^+=\{x \in \Bbb R^{n+1}:  x \cdot R  \geq 0\}$. For any vector $p \in H_R^+$, define $\Delta_R(p)=H_R \cap C^-(p)$ to be the $n$-dimensional simplex in the hyperplane $H_R$. For the definitions of simplex and facet and their properties, we refer the reader to~\cite{Mat02, Sch98}. For simplicity we denote $\Delta_R(R)$ by $\Delta_R$.

It is easy to see that for any $p \in H^+_R$ there exists a unique $\lambda \geq 0$ and $p' \in H_R$ such that $p=p'+\lambda R$. Define the {\it projection} function $\pi:H^+_R \rightarrow H_R$ as follows: for any $p \in H_R^+$, define $\pi(p)=p'$. It is also easy to see that $\pi(p)=p-\lambda R$ where $\lambda={(p \cdot R) / \|R\|^2}$. We refer to $\pi(p)$ the {\it projection} of the point $p$ into the hyperplane $H_R$ along the vector $R$. The following lemma is an immediate consequence of the above definition.
\begin{lemma}
{
\label{delta_simplex_easy_lem}
If $p=(p_0, \dots, p_n) \in H^+_R$ and $p=\pi(p)+\lambda R$, then
\begin{itemize}
\item [(i)] $\Delta_R(p)=\pi(p)+\lambda\Delta_R$.
\item [(ii)] $F_i=\Delta_R(p) \cap \{x \in \Bbb R^n: x_i=p_i\}$ for all $0 \leq i \leq n$, defines all the {\it facets} of the simplex $\Delta_R(p)$.
\end{itemize}
}
\end{lemma}
It is easy to see that $\Delta_R$ is the simplex in $H_R$ with vertices $b^0, \dots, b^n \in H_R$ whose coordinates are:
$$b^i_j=\left\{{\begin{matrix} -\sum_{k \neq i}{r^2_k \over r_i} & \hbox{ if } i = j \cr r_i & \hbox { otherwise } \end{matrix}}\right.$$
for all $0 \leq j \leq n$.
\begin{definition}
\label{delta_distance_def}
For any two points $p,q \in H_R$, define the ${\Delta_R}$-distance function between $p$ and $q$ as follows:
$$d_{\Delta_R}(p,q)=\inf\{\lambda \geq 0: q \in p+\lambda \Delta_R\}.$$
\end{definition}
The $\Delta_R$-distance function defined above is a {\it gauge function} (which is often used in the study of convex bodies). For more on gauge functions and their properties, see~\cite{Sie89}. \\For any point $p \in \L$ define $d_{\Delta_R}(p,\L)=\min\{\lambda \geq 0: \hbox{ there exists } \, q \in \L \hbox{ such that }q \in p+\lambda \Delta_R\}$.

The following remark can be considered as a generalization of Lemma 4.7 in~\cite{AM09}, and its proof easily follows from Definition~\ref{delta_distance_def}.
\begin{remark}
{
Given any two vectors $p,q \in H_R$, $$d_{\Delta_R}(p,q)=\max_{0 \leq i \leq n}\{{q_i-p_i \over r_i}\}.$$
}
\end{remark}
\begin{proof}
{
By Definition~\ref{delta_distance_def}, $$d_{\Delta_R}(p,q)=\inf\{\lambda \geq 0: q \in p+\lambda \Delta_R\}=\inf\{\lambda \geq 0 : q \in p+C^-({\lambda R})\}$$$$=\inf\{\lambda \geq 0 : q \leq p+\lambda R\}=\max_{0 \leq i \leq n}\{{q_i-p_i \over r_i}\}.$$
}
\end{proof}
\begin{definition}
\label{extreme_critical_def}
Define
\begin{eqnarray*}
 Ext(\Sigma(\Lambda)) &=&\{\nu \in \Sigma(\Lambda): \deg_R(\nu) \geq \deg_R(p), {\hbox{ for all }} p \in N(\nu) \cap \Sigma(\Lambda)\}, \\
   Ext(\overline{\Sigma}_{\Bbb R}(\Lambda)) &=& \{\nu \in \overline{\Sigma}_{\Bbb R}(\Lambda): \exists \, \delta>0, \hbox{ such that } \deg_R(\nu) \geq \deg_R(p), {\hbox{ for all }} p \in B(\nu,\delta) \cap \overline{\Sigma}_{\Bbb R}(\Lambda)\}, \\
   Crit(\L) &=& \{\nu \in H_R: \exists \, \delta>0 \hbox{ such that } d_{\Delta_R}(\nu,\L) \geq d_{\Delta_R}(p,\L), \hbox{ for all } p \in B(\nu, \delta) \cap H_R\}.
\end{eqnarray*}
where $N(\nu)$ consists of all points $D \in \Bbb Z^{n+1}$ such that $\|D-\nu\|_{\onev} \leq 1$.
We call $Ext(\Sigma(\Lambda))$, $Ext(\overline{\Sigma}_{\Bbb R}(\Lambda))$ and $Crit(\L)$, the set of {\it extreme points} or {\it extreme divisors}  of $\Sigma(\Lambda)$, $\overline{\Sigma}_{\Bbb R}(\Lambda)$ and the set of critical points of $\L$, respectively.
\end{definition}
\begin{lemma}
{
\label{cone_comparison_lem}
If $p,q \in H_R^{+}$, then $p\leq
q$ if and only if $\Delta_R(p) \subseteq \Delta_R(q)$.  In particular, $p<
q$ if and only if $\Delta_R(p) \subsetneq int(\Delta_R(q))$.
}
\end{lemma}
\begin{proof}
{
It is easy to see that $p \leq q$ if and only if $C^{-}(p) \subseteq
C^{-}(q)$.  Now the second part of Lemma~\ref{delta_simplex_easy_lem} implies that $C^{-}(p) \subseteq
C^{-}(q)$ if and only if
$(C^{-}(p) \cap H_R) \subseteq (C^{-}(q) \cap H_R)$.
}
\end{proof}
An easy application of Lemma~\ref{sigma_closure} is that if $p \in Ext(\overline{\Sigma}_{\Bbb R}(\Lambda))$, then $p \not \in \Lambda$. The following theorem characterizes the set of extreme points of $\overline{\Sigma}_{\Bbb R}(\Lambda)$.
\begin{theorem}
{
\label{extreme_sigma_closure_thm}
If $p \in \overline{\Sigma}_{\Bbb R}(\Lambda) \setminus \Lambda$ then $p \in Ext(\overline{\Sigma}_{\Bbb R}(\Lambda))$ if and only if
each facet of the simplex $\Delta_R(p)$ contains a point of $\Lambda$ in its interior.
}
\end{theorem}
\begin{proof}
{
Assume that $p=(p_0, \dots, p_n) \in \overline{\Sigma}_{\Bbb R}(\Lambda) \setminus \Lambda$. Let $F_i$, $0 \leq i \leq n$ be the facets of $\Delta_R(p)$. Let $0 \leq i \leq n$ be such that $int(F_i)$ contains no point of $\Lambda$. By Lemma~\ref{delta_simplex_easy_lem} (ii), there exists an $\epsilon >0$ such that $\Delta_R(p+\epsilon e_i)$ does not contain
any points of $\Lambda$ in its interior. Hence Lemma~\ref{cone_comparison_lem} and Lemma~\ref{sigma_closure} imply that $p+\epsilon e_i \in \overline{\Sigma}_{\Bbb R}(\Lambda)$. Since $\deg_R(p) < \deg_R(p+\epsilon e_i)$, the point $p$ is not an extreme point.

Conversely, assume that $p \in \overline{\Sigma}_{\Bbb R}(\Lambda) \setminus \Lambda$ is such that the interior of each facet $F$ of $\Delta_R(p)$ contains a point of $\Lambda$. We claim that for any $v=(v_0, \dots, v_n) \in \Bbb R^{n+1}$, either  $\deg_R(p+\epsilon v) \leq \deg_R(p)$ for all $\epsilon \geq 0$, or there exists $\lambda>0$ such that for all $0< \epsilon \leq \lambda$, $p+\epsilon v \not \in \overline{\Sigma}_{\Bbb R}(\Lambda)$. If $v \leq {\zerov}$, then for all $\epsilon \geq 0$, $\deg_R(p+\epsilon v) \leq \deg_R(p)$. Now, without loss of generality assume that $v_0 >0$ and $v_1  \leq 0$. Suppose $x \in int(F)$ where $F=\Delta_R(D) \cap \{y \in \Bbb R^n: (y-D) \cdot e_0=0\}$. Since $x \in int(F)$, we can pick $\lambda>0$ small enough such that for all $0 < \epsilon \leq \lambda $, $x \in int(\Delta_R(p+\epsilon v))$. Thus Lemma~\ref{cone_comparison_lem} and Lemma~\ref{sigma_closure} imply that $x \not \in \overline{\Sigma}_{\Bbb R}(\Lambda)$ for all $0 < \epsilon \leq \lambda$. This completes the proof of the claim. It is easy to see that the proof of the theorem follows from the claim.
}
\end{proof}
\begin{corollary}
{
\label{extreme_sigma_closure_integer_cor}
$ Ext(\overline{\Sigma}_{\Bbb R}(\Lambda)) \subset \Bbb Z^{n+1}$.
}
\end{corollary}
\begin{proof}
{
Let $p\in Ext(\overline{\Sigma}_{\Bbb R}(\Lambda))$. Theorem~\ref{extreme_sigma_closure_thm} shows that the interior of every
facet $F$ of $\Delta_R(p)$ contains a point of $\Lambda$. Since $\Lambda \subseteq \Bbb Z^{n+1}$, the second part of Lemma~\ref{delta_simplex_easy_lem} implies that $p \in \Bbb Z^{n+1}$.
}
\end{proof}
\begin{theorem}
{
\label{extreme_sigma_closure_sigma_thm}
A divisor $\nu \in Ext(\Sigma(\Lambda))$ if and only if $\nu+{\onev} \in Ext(\overline{\Sigma}_{\Bbb R}(\Lambda))$.
}
\end{theorem}
\begin{proof}
{
Corollary~\ref{extreme_sigma_closure_integer_cor} implies that $Ext(\overline{\Sigma}_{\Bbb R}(\Lambda)) \subseteq \Bbb Z^{n+1}$. The theorem immediately follows from Lemma~\ref{sigma_closure_sigma_lem}.
}
\end{proof}
The set of critical points of $\L$ ($Crit(\L)$ in Definition~\ref{extreme_critical_def}) is the set of local maxima of the function $d_{\Delta_R}(\cdot,\L)$. The following theorem characterizes  critical points of $\L$ in terms of extreme points of $\overline{\Sigma}_{\R}(\L)$.
\begin{theorem}
{
\label{extreme_sigma_closure_critical_thm}
For $p \in H_R$, let $\lambda=d_{\Delta_R}(p,\L)$ and $p'=p+\lambda R$.  Then $p' \in Ext(\overline{\Sigma}_{\R}(\L))$ if and only if $p \in Crit(\L)$.
}
\end{theorem}
\begin{proof}
{
If $p' \in Ext(\overline{\Sigma}_{\R}(\L))$ then by Theorem~\ref{extreme_sigma_closure_thm} each facet of the simplex $\Delta_R(p+\lambda R)=p+\lambda \Delta_R$ contains a point of $\L$ in its interior. This shows that $p \in Crit(\L)$.

Conversely, assume that $p \in Crit(L)$and $p' \not \in Ext(\overline{\Sigma}_{\R}(\L))$. As the proof of Theorem~\ref{extreme_sigma_closure_thm} shows, there exist $0 \leq i \leq n$ and  $\delta>0$ such that for all $0<\epsilon \leq \delta$, $p'_{\epsilon}=p'+\epsilon e_i \in \overline{\Sigma}_{\R}(\L)$. For each $0<\epsilon \leq \delta$, let $p_{\epsilon}=\pi(p'_{\epsilon})$ to be the projection of $p'_{\epsilon}$ along $R$ into $H_R$. Lemma~\ref{distance_HR_sigma_closure_lem} implies that $d_{\Delta_R}(p_{\epsilon},\L)=\left({p'_{\epsilon} \cdot R \over \|R\|^2}\right)$. Since $p_{\epsilon}' \cdot R > p' \cdot R$, we conclude that $d_{\Delta_R}(p_{\epsilon},\L) > d_{\Delta_R}(p,\L)$, a contradiction.
}
\end{proof}
\begin{corollary}
{
\label{extreme_L_ciritical_cor}
Let $\varphi: Ext(\Sigma(\L)) \rightarrow Crit(\L)$ be as follows: For any $\nu \in Ext(\Sigma(\L))$, $\varphi(\nu)=\pi(\nu+{\onev})$. Then $\varphi$ is a bijection.
}
\end{corollary}
\begin{proof}
{
This follows from Theorems~\ref{extreme_sigma_closure_critical_thm} and ~\ref{extreme_sigma_closure_sigma_thm}.
}
\end{proof}
\begin{lemma}
{
\label{distance_HR_sigma_closure_lem}
Let $p \in H_R$, $\lambda=d_{\Delta_R}(p,\L)$ and $\lambda'=\max\{t \geq 0: p+tR \in \overline{\Sigma}_{\Bbb R}(\L)\}$. Then $\lambda=\lambda'$.
}
\end{lemma}
\begin{proof}
{
First note that since $p \in \overline{\Sigma}_{\Bbb R}(\L)$ and $\overline{\Sigma}_{\Bbb R}(\L)$ is a closed set, $\max\{t \geq 0: p+tR \in \overline{\Sigma}_{\Bbb R}(\L)\}$ is well-defined. The first part of Lemma~\ref{delta_simplex_easy_lem} implies that $p+t\Delta_R=\Delta_R(p+tR)$. Now, for all $0 \leq t \leq \lambda$, by applying Lemma~\ref{sigma_closure} and Lemma~\ref{cone_comparison_lem}, we conclude that $p+tR \in \overline{\Sigma}_{\Bbb R}(\L)$. So $\lambda' \geq \lambda$. Conversely, suppose $t \geq 0$ is such that $\L \cap (p+t\Delta_R) \neq \emptyset$. Lemma~\ref{sigma_closure} and Lemma~\ref{cone_comparison_lem} imply that $p+tR \in \overline{\Sigma}_{\Bbb R}(\L)$ if and only if $\L \cap int(p+t\Delta_R) = \emptyset$. This shows that $\lambda' \leq \lambda$, completing the proof of the lemma.
}
\end{proof}

\begin{lemma}
{
\label{dominate_by_extreme_lem}
There exists a constant $C$ depending only on the lattice $\L$ and the vector $R$ such that for any point $p \in \Sigma(\Lambda)$, we have:
\begin{enumerate}
\item[(i)] $\deg_R(p) \leq C$,
\item[(ii)] there exists some $\nu \in Ext(\Lambda)$
such that $p \leq \nu$.
\end{enumerate}
}
\end{lemma}
\begin{proof}
{
$(i)$: First, we claim that there exists $c$ such that for all $p\in H_R,
d_{\Delta_R}(p, \Lambda)\leq c.$
We start by noting that there exists a constant $K$ depending only on $R$ such that
$d_{\Delta_R}(p,q) \leq K\cdot \|p-q\|$. This follows immediately
by letting the constant $K$ be the largest radius of a sphere in $H_R$ with center at the origin
contained in $\Delta_R$.

Let $\{l_0, ..., l_{n-1}\}$ be a set of
generators of $\Lambda$, and let $P$ be the parallelotope generated
by $l_0,... l_{n-1}$. Because the $\Delta_R$-distance function is invariant
under translation by lattice points, it is sufficient to prove the claim for all $p \in P$. By letting $c$ be $K$ times the maximum $\ell^2$-distance from a point
in $P$ to the vertices of $P$ (diameter of $P$ by $\ell^2$-norm), the claim is proved.

To prove the first part, it is enough to show that for all $p \in H^+_R \cap \Sigma(\L)$, $\deg_R(p) \leq C$. Let $p'=\pi(p)$, $\lambda \geq 0$ be such that $p=p' +\lambda R$.  Lemma~\ref{cone_comparison_lem} implies that $p \in \Sigma (\Lambda)$ if and only if $\Delta_R (p)$ contains no points of $\Lambda$. Lemma~\ref{distance_HR_sigma_closure_lem} and Theorem~\ref{extreme_sigma_closure_sigma_thm} imply that $\lambda \leq dist_{\Delta_R}(p,\L)$, so $\lambda \leq c$. Therefore, $\deg_R(p) = \lambda \|R\|^2 \leq c \|R\|^2$. This shows that $C \leq c \|R\|^2$, which completes the proof of the first part.

$(ii)$: Let $p\in \Sigma(\Lambda)$. The first part shows that the degrees of points in $Ext(\Lambda)$
are bounded above by $C$. Therefore $C^+(p) \cap \Sigma(\L)$ is a finite set. This immediately shows that there exists $\nu \in Ext(\L)$ such that $p \leq \nu$. To be more precise, one can find an extreme point $\nu \in Ext(\L)$ greedily by starting at point $p$ and walking in positive directions as much as possible.
}
\end{proof}
\begin{lemma}
{
\label{rank_degree_plus_lem}
For any divisor $D\in \Z^{n+1}$, $r(D)=\min \{\deg_R^+(D-\nu) : \nu \in Ext(\Lambda)\}-1$.
}
\end{lemma}
\begin{proof}
{
First we show that $\min \{\deg_R^+(D-\nu) : \nu \in Ext(\Lambda)\}\leq r(D)+1$.
Let $E \geq {\zerov}$ with $\deg_R(E)=r(D)+1$ be such that $D-E \in \Sigma (\Lambda)$, where the existence of $E$ guaranteed by Lemma~\ref{distance_rank_lem}.
By Lemma~\ref{dominate_by_extreme_lem}, there exists $\nu \in \Sigma(\L)$ such that $\nu \geq D-E$. Let $E' = \nu - (D-E)$. We claim that
$E' \cdot E = 0$.  Suppose not and assume there exists $0 \leq i \leq n$ such that $E_i, E'_i
\geq 1$. Note that $D- (E- e_i) \in \Sigma (\Lambda)$ as $\nu \geq D-
(E- e_i)$, but $\deg_R(E- e_i)<\deg_R(E)=r(D)+1$, a contradiction. This gives
that $\deg_R^+(D-\nu)= \deg_R^+(E-E')=\deg(E)=r(D)+1$.

For proving the reverse inequality, let $\nu \in Ext(\Lambda)$ be such that
$\deg^+(D-\nu)$ is minimum.  Because $\nu \geq \nu + (D-\nu)^-=
D-(D-\nu)^+$, it follows that  $D-(D-\nu)^+ \in \Sigma(\Lambda)$. Hence Lemma~\ref{distance_rank_lem} implies that
$r(D)\leq \min\{\deg_R^+(D-\nu) : \nu \in Ext(\Lambda)\}-1$, which completes
the proof.
}
\end{proof}
\subsection{Riemann-Roch Theorem for Uniform and Reflection Invariant Sub-lattices of $\L_R$}
\begin{definition}
{
Let $\L$ be a sub-lattice of $\L_R$ of rank $n$, and $Ext(\Sigma(\L))$ be the set of extreme points of $\Sigma(\L)$. Define
\begin{eqnarray*}
 g_{\min} &=&\min\{\deg_R(\nu) : \nu \in Ext(\Sigma(\Lambda))\}+1, \\
 g_{\max} &=&\max\{\deg_R(\nu) : \nu \in Ext(\Sigma(\Lambda))\}+1.
\end{eqnarray*}
We say the lattice $\L$ is uniform if $g_{\min}=g_{\max}$.
}
\end{definition}
\begin{definition}
{
Let $\L$ be a sub-lattice of $\L_R$ of rank $n$. We say $\L$ is reflection invariant if $-Crit(\L)$ is a translate of $Crit(\L)$, i.e., if there exists $v \in \R^{n+1}$ such that $-Crit(\L)=Crit(\L)+v$.
}
\end{definition}
\begin{definition}
\label{canonical_def}
{
Let $\L$ be a sub-lattice of dimension $n$ of $\L_R$. We say a divisor $K \in \Z^{n+1}$ is a canonical divisor of $\L$, or equivalently $\L$ has a canonical divisor $K$, if for all divisors $D \in \Z^{n+1}$,
$$\deg_R(D)-3g_{\max}+2g_{\min}+1 \leq r(D)-r(K-D) \leq \deg_R(D)- g_{\min}+1.$$
}
\end{definition}
\begin{lemma}
\label{bijection_min_lem}
Suppose $\phi:\mathcal{A} \rightarrow \mathcal{A}'$ is a bijection between sets, and $f:\mathcal{A} \rightarrow \Z$ and $f':\mathcal{A}' \rightarrow \Z$ are functions whose values are bounded from below. If there exist constants $c_1,c_2 \in \Z$ such that for all $a \in \mathcal{A}$, $$c_1 \leq f(a)-f'(\phi(a)) \leq c_2,$$ then
$$ c_1 \leq \min_{a \in \mathcal{A}} f(a) - \min_{a' \in \mathcal{A'}} f'(a') \leq c_2.$$
\end{lemma}
\begin{proof}
Since $f$ and $f'$ are integer valued functions whose values are bounded from below, there exists $x \in \mathcal{A}$ and $y \in \mathcal{A}'$ such that $f(x)=\min_{a \in \mathcal{A}} f(a)$ and $f'(y)=\min_{a' \in \mathcal{A'}} f'(a')$. The choice of $x$ and $y$ implies that $f(x)-f'(y) \leq f(\phi^{-1}(y))-f'(y) \leq c_2$, and $f(x)-f'(y) \geq f(x)-f'(\phi(x)) \geq c_1$. Hence $c_1 \leq f(x)-f'(y) \leq c_2$, as desired.
\end{proof}
\begin{theorem}
\label{reflection_invariant_inequality}
Let $\L$ be a reflection invariant sub-lattice of $\L_R$ of rank $n$. Then $\L$ has a canonical divisor, i.e. there exists a divisor $K$ such that for all $D \in \Bbb Z^{n+1}$,
$$ \deg_R(D)-3g_{\max}+2g_{\min}+1 \leq r(D)-r(K-D) \leq \deg_R(D)- g_{\min}+1.$$
\end{theorem}
\begin{proof}
First we construct the canonical divisor $K$ and then we show it has the desired property. Since $\L$ is reflection invariant, there exists a vector $v \in \R^{n+1}$ such that $-Crit(\L)=Crit(\L)+v$. Therefore there exists a bijection function $\eta$ from $Crit(\L)$ to itself such that $\eta(c)+c=v$. Let $\varphi:Ext(\Sigma(\L)) \rightarrow Crit(\L)$ be the bijection described in Corollary~\ref{extreme_L_ciritical_cor}. Define the bijection $\phi$ from $Ext(\Sigma(\L))$ to itself so that for all $\nu \in Ext(\Sigma(\L))$, $\phi(\nu)=\varphi^{-1}\eta\varphi(\nu)$. Since for all $\nu \in Ext(\Sigma(\L))$, $\deg_R(\nu+\phi(\nu)) \leq 2g_{\max}$, there exists $\nu_0 \in Ext(\Sigma(\L))$ such that $\deg_R(\nu_0+\phi(\nu_0))$ is as large as possible. Let the canonical divisor $K$ be $\nu_0+\phi(\nu_0)$.

For any $\nu \in Ext(\Sigma(\L))$, let $c=\varphi(\nu)$; then we have:
$$\phi(\nu)+\nu=\phi(\varphi^{-1}(c))+\varphi^{-1}(c)=\varphi^{-1}\eta(c)+\varphi^{-1}(c)=\lambda R+v-2 \times \onev,$$
where $\lambda \in \R$ is a constant depends on $\nu$ (or equivalently $c$). Hence, the choice of $K$ implies that for any $\nu \in Ext(\Sigma(\L))$, there exists $E_{\nu} \in \R^{n+1}_+$ such that $\phi(\nu)+\nu+E_{\nu}=K$. Therefore, for all divisor $D \in \Z^{n+1}$ and $\nu \in Ext(\Sigma(\L))$ we have:
\begin{eqnarray*}
 \deg^+_R(D-\nu) - \deg^+_R(K-D-\phi(\nu))  &=&\deg^+_R(D-\nu) - \deg^+_R(\phi(\nu)+\nu+E_{\nu}-D-\phi(\nu)) \\
   &=& \deg^+_R(D-\nu) - \deg^+_R(\nu+E_{\nu}-D)\\
   & \leq & \deg^+_R(D-\nu) - \deg^+_R(\nu-D) \\
   & = & \deg_R(D)-\deg_R(\nu)\\
   & \leq & \deg_R(D)-g_{\min}+1.
\end{eqnarray*}
Note that for all $\nu \in Ext(\Sigma(\L))$, $E_\nu=K-(\nu+\phi(\nu)) \leq 2g_{\max}-2g_{\min}$. Hence,
\begin{eqnarray*}
 \deg^+_R(D-\nu) - \deg^+_R(K-D-\phi(\nu))  &=&\deg^+_R(D-\nu) - \deg^+_R(\phi(\nu)+\nu+E_{\nu}-D-\phi(\nu)) \\
   &=& \deg^+_R(D-\nu) - \deg^+_R(\nu+E_{\nu}-D)\\
   & \geq & \deg^+_R(D-\nu) - \deg^+_R(\nu-D) - 2(g_{\max}-g_{\min}) \\
   & = & \deg_R(D)-\deg_R(\nu)-2g_{\max}+2g_{\min}\\
   & \geq & \deg_R(D)-3g_{\max}+2g_{\min}+1.
\end{eqnarray*}
Therefore for all $D \in \Z^{n+1}$ and all $\nu \in Ext(\Sigma(\L))$,
$$\deg_R(D)-3g_{\max}+2g_{\min}+1 \leq \deg^+_R(D-\nu) - \deg^+_R(K-D-\varphi(\nu)) \leq \deg_R(D)-g_{\min}+1.$$
For a fixed $D \in \Z^{n+1}$, $\deg_R(D)-3g_{\max}+2g_{\min}+1$ and $ \deg_R(D)-g_{\min}+1$ are constant integers, $\deg^+_R(D-\nu)$ and $\deg^+_R(K-D-\varphi(\nu))$ are integer value functions bounded from below by zero, and $\varphi$ is a bijection from $Ext(\Sigma(\L))$ to itself, hence Lemma~\ref{bijection_min_lem} implies that
$$\deg_R(D)-3g_{\max}+2g_{\min}+1 \leq \min_{\nu \in Ext(\Sigma(\L))}\deg^+_R(D-\nu) - \min_{\nu \in Ext(\Sigma(\L))}\deg^+_R(K-D-\nu) \leq \deg_R(D)-g_{\min}+1.$$
The assertion of the theorem now follows from Lemma~\ref{rank_degree_plus_lem}.
\end{proof}
\begin{definition}
{
Let $\L$ be a uniform sub-lattice of dimension $n$ of $\L_R$. We say $\L$ has the Riemann-Roch property if there exists a divisor $K$ with degree $2g-2$, where $g=g_{\min}=g_{\max}$, such that for all divisor $D \in \Z^{n+1}$:
$$r(D)-r(K-D)=\deg(D)-g+1.$$
}
\end{definition}
\begin{theorem}
\label{RI_equiv_RR_prop_thm}
Let $\L$ be a uniform sub-lattice of dimension $n$ of $\L_R$. Then $\L$ is reflection invariant if and only if $\L$ has the Riemann-Roch property.
\end{theorem}
\begin{proof}
Assume $\L$ is reflection invariant and let $K$ be the canonical divisor obtained in the proof of Theorem~\ref{reflection_invariant_inequality}. By applying Theorem~\ref{reflection_invariant_inequality}, its enough to show that $\deg(K)=2g-2$. The construction of $K$ shows that $K=\nu + \phi(\nu)$, where $\phi$ is the bijection obtained in proof of Theorem~\ref{reflection_invariant_inequality}. Since $\L$ is uniform, $g_{\min}=g_{\max}=g$. Hence $\deg_R(\nu)=\deg_R(\phi(\nu))=g-1$ and this implies that $\deg_R(K)=2g-2$.

Now, assume that $\L$ has the Riemann property. Assume $\nu$ is an extreme divisor of $\Sigma(\L)$, so the first part of Lemma~\ref{distance_rank_lem} implies that $r(\nu)=-1$. Since $\L$ is uniform $\deg_R(\nu)=g-1$ and this shows that $r(K-\nu)=r(\nu)=-1$. By Lemma~\ref{distance_rank_lem}, $K-\nu \in \Sigma(\L)$, and is hence an extreme divisor of $\Sigma(\L)$. Hence the function $\psi$ defined as $\psi(-\nu)=K-\nu$, for all $\nu \in Ext(\L)$ is a bijection from $Ext(\L)$ to itself. If $\varphi$ is the function defined in Corollary~\ref{extreme_L_ciritical_cor}, the function $\varphi o \psi o \varphi^{-1}$ is a bijection from $Crit(\L)$ to itself. It is easy to see that for any $p \in Crit(\L)$, $\varphi(\psi(\varphi^{-1}(p)))=-p+\pi(K)+2\pi({\onev})$, and by picking $v=-\pi(K)-2\pi({\onev})$, we have $-Crit(\L)=Crit(\L)+v$.
\end{proof}
\begin{definition}
{
We say a sub-lattice $\L$ of $\L_R$ has the Riemann-Roch formula if there exists a an integer $m \in \Z$ and a divisor $K$ of degree $2m-2$ such that for all $D \in \Z^{n+1}$:
$$r(D)-r(K-D)=\deg_R(D)-m+1.$$
}
\end{definition}
\begin{theorem}
\label{RR_formula_equiv_U_RI_thm}
Let $\L$ be a sub-lattice of dimension $n$ of $\L_R$. Then $\L$ has a Riemann-Roch formula if and only if $\L$ is uniform and reflection invariant, in particular $\L$ has the Riemann-Roch property.
\end{theorem}
\begin{proof}
If $\L$ is uniform and reflection invariant, then Theorem~\ref{RI_equiv_RR_prop_thm} implies that $\L$ has Riemann-Roch property and therefore $\L$ has the Riemann-Roch formula with $m=g_{\max}$.

For proving the other direction it is enough by Theorem~\ref{RI_equiv_RR_prop_thm} to show that $\L$ is uniform and $m=g_{\max}$. First, we show that $m=g_{\max}$. Let $D$ be a divisor with $\deg_R(D) \geq m$. The Riemann-Roch formula implies that $r(D)-r(K-D) \geq 1$ and since $r(K-D) \geq -1$, we have $r(D) \geq 0$. It follows that $g_{\max} \leq m$.

We know that for any divisor $D \in \Z^{n+1}$, if the degree of $D$ is more that $g_{\max}-1$ then the divisor is effective, so $\deg_R(D)-r(D) \leq g_{\max}$. On the other hand, if $\deg_R(D) >2m-2$, then $\deg_R(K-D) < 0$, therefore $r(K-D)=-1$. The Riemann-Roch formula implies that $\deg(D)-r(D) = m$. Therefore, $m \leq g_{\max}$. This shows that $m=g_{\max}$.

To prove uniformity, let $\nu \in Ext(\Sigma(\L))$ and $\deg_R(\nu) < g_{\max}-1$. Since $\deg_R(K)=2g_{\max}-2$, $\deg_R(K-\nu) \geq g_{\max}$, so $K-\nu \not \in \Sigma(\L)$, and by Lemma~\ref{distance_rank_lem} is equivalent to an effective divisor. The Riemann-Roch formula implies that $r(K-\nu)=g_{\max}-\deg(\nu)-2$, so there exists an effective divisor $E$ of degree $g_{\max}-\deg(\nu)-1>0$ such that $|K-\nu-E|=\emptyset$. We claim that $\nu+E$ is not equivalent to an effective divisor. The Riemann-Roch formula implies that $r(\nu+E)-r(K-\nu-E)=\deg_R(\nu+E)-g_{\max}+1=0$ and therefore $r(\nu+E)=-1$. By Lemma~\ref{distance_rank_lem}, $\nu+E \in \Sigma(\L)$, contradicting the fact that $\nu \in Ext(\Sigma(\L))$.
\end{proof}
\subsection{Riemann-Roch Theorem for sub-lattice of $\L_R$ and $\L_{\onev}$}
Let $R=(r_0, \dots, r_n) \in \N^{n+1}$ and $\mathcal{R}=diag(r_0, \dots, r_n)$ be a matrix mapping $\L_R$ to $\L_{\onev}$.
To be more precise, for any $p \in \L_R$ the image of $p$ is $\mathcal{R}p$.
For any set $S \subseteq \R^{n+1}$, let $\mathcal{R}S$ denote the set $\{\mathcal{R}p: p \in S\}$.
It is easy to see that if $\L \subseteq \L_R$ is a sub-lattice of dimension $n$ then $\mathcal{R}\L$ is a sub-lattice of $\L_{\onev}$ of dimension $n$.
\begin{lemma}
\label{lem:sigma_R_one}
Let $\L$ be a sub-lattice of dimension $n$ of $\L_R$. Then $\mathcal{R}\Sigma(\L)=\Sigma(\mathcal{R}\L)$.
\end{lemma}
The proof of above lemma follows easily from Definition~\ref{sigma_region_def} and the fact that $\mathcal{R}$ is an invertible matrix with positive diagonal entries.
\begin{lemma}
\label{lem:extremesigmaclosure_R_one}
Let $\L$ be a sub-lattice of dimension $n$ of $\L_R$. Then $\mathcal{R}Ext(\overline{\Sigma}_{\R}(\L))=Ext(\overline{\Sigma}_{\R}(\mathcal{R}\L))$.
\end{lemma}
\begin{proof}
Let $\nu \in Ext(\overline{\Sigma}_{\R}(\L))$ so that there exists some $\delta >0$ such that for all $p \in B(\nu, \delta) \cap \overline{\Sigma}_{\R}(\L)$, $\deg_R(\nu) \geq \deg_R(p)$. Let $\delta' = \delta$.  It is easy to see that if $q \in B(\mathcal{R}\nu, \delta')$, we have $\mathcal{R}^{-1}q \in B(\nu, \delta)$. Hence
$\deg_R(\mathcal{R}^{-1}q) \leq \deg_R(\nu)$ and therefore $\deg_{\onev}(q) \leq \deg_{\onev}(\mathcal{R}\nu)$. Here we have used the fact that for any $D \in \Z^{n+1}$, $\deg_R(D)=\deg_{\onev}(\mathcal{R}D)$ and Lemma~\ref{lem:sigma_R_one}. This proves that $\mathcal{R}Ext(\overline{\Sigma}_{\R}(\L)) \subseteq Ext(\overline{\Sigma}_{\R}(\mathcal{R}\L))$. The other direction is proved similarly.
\end{proof}
The following corollary immediately follows from Lemma~\ref{lem:extremesigmaclosure_R_one} and Theorem~\ref{extreme_sigma_closure_sigma_thm}.
\begin{corollary}
\label{cor:uniformity_R_one}
Let $\L$ be a sub-lattice of dimension $n$ of $\L_R$. Then $\L$ is uniform if and only if $\mathcal{R}\L \subseteq \L_{\onev}$ is uniform.
\end{corollary}
\begin{lemma}
\label{lem:reflection_R_one}
Let $\L$ be a uniform sub-lattice of dimension $n$ of $\L_R$. Then $\L$ is reflection invariant if and only if $\mathcal{R}\L \subseteq \L_{\onev}$ is reflection invariant.
\end{lemma}
\begin{proof}
First suppose $\L$ is reflection invariant. Then there exists a vector $v \in \R^{n+1}$ such that $-Crit(\L)=Crit(\L)+v$. By applying Lemma~\ref{lem:extremesigmaclosure_R_one} and Theorem~\ref{extreme_sigma_closure_critical_thm}, let $\mathcal{R}\nu-\onev-\deg_{\onev}(R\nu - \onev)\onev$ be an arbitrary point of $Crit(\mathcal{R}\L)$ where $\nu$ is an arbitrary point of  $Ext(\overline{\Sigma}_{\R}(\L))$.
Now, by applying Theorem~\ref{extreme_sigma_closure_critical_thm}, $$\nu-\onev-\deg_R(\nu-\onev)R \in Crit(\L).$$ Since $\L$ is reflection invariant, there exists $\nu' \in Ext(\overline{\Sigma}_{\R}(\L))$ such that
$$-\nu+\onev+\deg_R(\nu-\onev)R=\nu'-\onev-\deg_R(\nu'-\onev)R+v,$$
therefore
$$-\mathcal{R}\nu+\mathcal{R}\onev+\deg_R(\nu-\onev)\mathcal{R}R=\mathcal{R}\nu'-\mathcal{R}\onev-\deg_R(\nu'-\onev)\mathcal{R}R+\mathcal{R}v.$$
Since $\L$ is uniform $\deg_R(\nu-\onev)$ is a constant independent from the choice of  $\nu \in Ext(\overline{\Sigma}_{\R}(\L))$. Hence, $\mathcal{R}\nu-\mathcal{R}\nu'=u$ where $u$ is constant vector in $\R^{n+1}$ which does not depend on $\nu$ or $\nu'$. Since $\mathcal{R}\L$ is uniform, $\deg_{\onev}(\mathcal{R}\nu-\onev)$ is a constant independent from the choice of $\nu \in Ext(\overline{\Sigma}_{\R}(\L))$. This shows that
$$\mathcal{R}\nu-\mathcal{R}\nu'=u+2\deg_{\onev}(\mathcal{R}\nu-\onev)+2\times \onev.$$
Hence $\mathcal{R}\L$ is reflection invariant. The other direction is proved similarly.
\end{proof}
Recall the definition of the canonical vector (Definition~\ref{canonical_def}) and the argument in the proof of Lemma~\ref{reflection_invariant_inequality} in constructing a canonical vector for a reflection invariant sublattice of $\L_R$. So we can consider the following corollary as a consequence of Theorem~\ref{extreme_sigma_closure_sigma_thm}, Lemma~\ref{lem:extremesigmaclosure_R_one}, and Lemma~\ref{lem:reflection_R_one}.
\begin{corollary}
\label{cor:canonical_1_R}
Let $\L$ be a reflection invariant sub-lattice of dimension $n$ of $\L_{R}$. If $K$ is a canonical vector of $\mathcal{R}\L$ then $\mathcal{R}^{-1}(K+2\times \onev)-2\times \onev$ is a canonical vector of $\L$.
\end{corollary}
The following theorem immediately follows from Theorem~\ref{RR_formula_equiv_U_RI_thm}, Corollary~\ref{cor:uniformity_R_one} and Lemma~\ref{lem:reflection_R_one}.
\begin{theorem}
\label{thm:RR_R_one}
Let $\L$ be a uniform sub-lattice of dimension $n$ of $\L_R$. Then $\L$ has the Riemann-Roch property if and only if $\mathcal{R}\L \subseteq \L_{\onev}$ has the Riemann-Roch property.
\end{theorem}
\section{Chip-Firing Game on Directed Graphs \label{chip_sec}}
\subsection{Row Chip-Firing Game, The Sandpile Model and Riemann-Roch Theory}
Let $\vec{G}$ be a directed graph with vertex set $\{v_0, ..., v_n \}$ and adjacency matrix $\vec{A}$ whose entry $\vec{A}_{i,j}$ for $0\leq i,j \leq n$ is the number of edges directed from $v_i$ to $v_j$.  Let $\vec{\mathcal{D}}=diag(\deg^+(v_0), \dots, \deg^+(v_n))$ where $\deg^+(v)$ denotes the number edges leaving vertex $v \in V(\vec{G})$. We call the matrix $\vec{Q}=\vec{\mathcal{D}}-\vec{A}$ the {\it Laplacian matrix} of the directed graph $\vec{G}$.  We define $\L_{\vec{G}}$ to be the lattice spanned by the rows of $\vec{Q}$.

In this section we study the following row chip-firing game on vertices of $\vec{G}$.  Begin with $D \in \Z^{n+1}$, which we call a configuration or a divisor, whose $i$th entry $D(v_i)$ is the number of chips at vertex $v_i$. In each {\it move} of the game either a vertex {\it borrows} or {\it fires}. We say a vertex {\it fires} if it sends a chip along each of its outgoing edges to its neighbors and {\it borrows} if it receives a chip along each of its incoming edges from its neighbors. We say that a vertex is in {\it debt} if the number of chips at that vertex is negative. The objective of the game is to bring every vertex out of debt by some sequence of moves.  Note that the game is ``commutative'' in the sense that the order of firings and borrowings does not effect the final configuration.  For $f \in \Z^{n+1}$, we may interpret the divisor $D'=D-\vec{Q}^Tf$ as the divisor obtained from $D$ by a sequence of moves in which the vertex $v_i$ fires $f(v_i)$ times if $f(v_i) \geq 0$ and it borrows $f(v_i)$ times if $f(v_i) \leq 0$. We refer to $f$ as a {\it firing strategy}. Note that both firing strategies and divisors are vectors in $\Z^{n+1}$. We say a configuration is a {\it winning configuration} if all of the vertices are out of debt. We call a sequence of moves which achieves a winning configuration a {\it winning strategy}. The question of whether a winning strategy exists is equivalent to the question of whether there exists a firing strategy $f\in \Z^{n+1}$ and an effective divisor $E \in \Z_{\geq 0}^{n+1}$ such that $E=D+\vec{Q}^T f$, i.e., $D-E \in \L_{\vec{G}}$, $|D| \neq \emptyset$ or $r(D)\geq 0$.  In what follows we will restrict our attention to strongly connected directed graphs.  The main motivation for this consideration is given in the following lemma which, interperetted combinatorially, characterizes strongly connected digraphs in terms of which firings leave a divisor unaffected.

\begin{lemma}
\label{lem:leftkernel}
A directed graph $\vec{G}$ is strongly connected if and only if there exists a vector $R \in \N^{n+1}$, unique up to multiplication by a real constant, such that ${\vec{Q}}^TR=0$.
\end{lemma}
\begin{proof}
{
Suppose $\vec{G}$ is strongly connected. For the sake of contradiction suppose there exists $R \not \geq 0$ such that  ${\vec{Q}}^TR=0$.  Let $V^+$ be the set of vertices of $\vec{G}$ such that $R(v) > 0$ for all $v \in V^+$. Let $D=\vec{Q}^TR$.  Since the net amount of chips leaving $V^+$ is positive,  there must exist some $v \in V^+$ such that $D(v)<0$, a contradiction. Now assume there exist two linearly independent firing strategies $R_1$ and $R_2$ then it is  easy to see that there exists a linear combination of $R_1$ and $R_2$, say $R$, such that $R \not \geq 0$. This proves the uniqueness. Note that we can take $R$ to be an integral vector.

Conversely, suppose $\vec{G}$ is not strongly connected. Let $V_1, \dots, V_t$ be the decomposition of vertices of $\vec{G}$ into maximal strongly connected components. Without loss of generality, let $V_1$ be a set of vertices such that there exists no edges from $u$ to $v$ where $u \in V_i$, $2 \leq i \leq t$ and $v \in V_1$. As above there exists $v \in V_1$ such that ${\vec{Q}}^TR(v)<0$, a contradiction.
}
\end{proof}
\subsubsection{\label{reduceddiv_sec}Reduced Divisors}
Let $f, f' \in \Bbb Z^{n+1}$ be firing strategies. We define an equivalence relation $\approx$ on $\Bbb Z^{n+1}$ by declaring $f \approx f'$ if $\vec{Q}^T(f-f')=\zerov$.
For any set $S \subseteq V(\vec{G})$, the {\it characteristic vector of} $S$, denoted by $\chi_{S}$, is the vector $\sum_{v_i \in S}e_i$.
We say a vector $f \in \Bbb Z^{n+1}$ is a {\it natural} firing strategy if $f \leq R$, and $f \not \leq {\zerov}$. We say a nonzero vector $f \in \Bbb Z^{n+1}$ is a {\it valid} firing strategy with respect to $v_0$ if $f(v_0)=0$, and ${\zerov} \leq f \leq R$.
The following lemma is an immediate consequence of Lemma~\ref{lem:leftkernel}.
\begin{lemma}
\label{natural_lemma}
Let $f \in \Bbb Z^{n+1}$ be a nonzero firing strategy then there exists a unique $f' \in \Bbb Z^{n+1}$ such that $f \approx f'$ and $f'$ is a natural firing strategy.
\end{lemma}
\begin{definition}
{
\label{reduced_def} Let $\vec{G}$ be a directed graph. We call a divisor $D$ $v_0$-reduced if the following two conditions hold:
\begin{enumerate}
\item[(i)] for all $v \in V(\vec{G})\setminus  \{v_0\}, D(v)\geq 0$,
\item[(ii)] for every valid firing $f$ with respect to $v_0$, there exists a vertex $v \in V(\vec{G}) \setminus \{v_0\}$ such that $(D-\vec{Q}^Tf)(v) <0$.
\end{enumerate}
}
\end{definition}
The following remark immediately follows from Definition~\ref{reduced_def}.
\begin{remark}
\label{reduced_rem}
If $D' \sim D$ is a $v_0$-reduced divisor then for all $k \in \Z$, $D'+k\chi_{\{v_0\}}$ is a $v_0$-reduced divisor and $D'+k\chi_{\{v_0\}} \sim D+k\chi_{\{v_0\}}$.
\end{remark}
\begin{lemma}
{
\label{firing_reduced_lemma}Let $D$ be a $v_0$-reduced divisor and let $f$ be a firing strategy such that $f(v_0) \leq 0$ and $f(v) >0$ for
some vertex $v \in V(\vec{G}) \setminus \{v_0\}$.  Then there exists $v \in V(\vec{G})\setminus \{v_0\}$ such that $(D-\vec{Q}^Tf)(v)<0$.
}
\end{lemma}
\begin{proof}
{
Lemma~\ref{natural_lemma} implies that there exists a natural firing strategy $f' \approx f$ with $f'(v_0)\leq f(v_0)=0$. Suppose $f^+$ and $f^-$ are the positive and negative part of $f'$. It is easy to see that $f^+$ is a valid firing strategy with respect to $v_0$. Hence there exists a vertex $v \in  V(\vec{G}) \setminus \{ v_0\}$ such that $(D-\vec{Q}^Tf^+)(v)<0$. Therefore,
$$(D-\vec{Q}^Tf)(v)= (D-\vec{Q}^Tf')(v)= (D-\vec{Q}^T f^+ -\vec{Q}^T f^-)(v) \leq (D- \vec{Q}^T f^+)(v) <0.$$
}
\end{proof}
\begin{lemma}
{
\label{reduce_exist_lemma} Let $\vec{G}$ be a directed graph and let $D$ be a divisor.  Then there exists a divisor $D' \sim D$ such that $D'$ is
$v_0$-reduced.
}
\end{lemma}
\begin{proof}
{
The proof that we present here is similar to the proof presented by Baker and Norine~\cite{BN07}(\S 3.1). The process of obtaining a $v_0$-reduced divisor $D' \sim D$ has two steps: first we bring every $v \in V(\vec{G}) \setminus \{v_0\}$ out of debt, so that it satisfies the first condition of Definition~\ref{reduced_def}, and then we ``reduce'' the divisor with respect to $v_0$, in order to satisfy the second condition of Definition~\ref{reduced_def}. For performing the first step, define $d(v)$, for all $v \in V(\vec{G}) \setminus \{v_0\}$, to be the length of the shortest directed path from $v_0$ to $v$. Let $d=\max_{v \in V(\vec{G}) \setminus \{v_0\}} d(v)$. For all $1 \leq i \leq d$, define $A_i=\{v \in V(\vec{G}): d(v)=i\}$. Now we bring the $A_i$'s out of debt consecutively, starting at $A_d$. We recursively define sequences of integers $b_i$ and divisors $D_i$ as follows.  Let $b_d=\max\left(\{-D(v): v \in A_d, D(v) \leq 0\} \cup \{0\}\right)$. Define $D_d=D-\vec{Q}^Tf_d$ where $f_d$ is the all zero vector except $f_d(v_j)=b_d$ if $v_j \not \in A_d$. It is easy to see that $D_d(v_j) \geq 0$ for all $v_j \in A_d$. Now suppose $1 \leq i \leq d-1$, and define $b_i=\max\left(\{-D(v): v \in A_i, D_{i+1}(v) \leq 0\} \cup \{0\}\right)$. Define $D_i=D_{i+1}-\vec{Q}^Tf_i$ where $f_i$ is the all zero vector except $f_i(v_j)=b_i$ if $v_j \not \in \bigcup_{k=i}^d A_k$. It is easy to see that $D_i(v_j) \geq 0$ for all $v_j \in A_i$ and $D_i(v_j)=D_{i+1}(v_j)$ for all $v_j \in \bigcup_{k=i+1}^d A_k$. Since $d$ is a finite number and the $b_i$'s are bounded,  the above procedure  terminates. It is easy to verify that $D_1 \sim D$ is a divisor such that no vertex other than $v_0$ is in debt. This completes the description of the first step.

Now we are going to explain the second step. Let $D'=D_1$ be the divisor obtained from the first step. While there exists a valid firing strategy $f$ with respect to $v_0$ such that $(D'-\vec{Q}^Tf)(v) \geq 0$ for all $v \in V(\vec{G}) \setminus \{v_0\}$, replace $D'$ by $D'-\vec{Q}^Tf$. If we show that the procedure terminates, it is obvious that $D'$ is a $v_0$-reduced divisor. Since $f(v_0)=0$ for any valid firing strategy with respect to $v_0$, the vertex $v_0$ must stop receiving money at some point. At this point none of its neighbors fires, so they must eventually stop receiving money. By iterating this argument we see that, since $v_0$ is reachable from every vertex, each vertex must stop receiving money at some point. Hence, the above procedure terminates at a $v_0$-reduced divisor.
}
\end{proof}

\begin{corollary}
{
\label{positive_divisor_reduction_cor}
Let $D$ be a divisor satisfying the property (i) in Definition~\ref{reduced_def}.  Then there exists a sequence of valid firings $f_1, \dots, f_k$ with respect to $v_0$ such that $D'=D-\vec{Q}^T(\sum_{i=1}^k f_i)$ is $v_0$-reduced.
}
\end{corollary}
%
\begin{lemma}
{
\label{r_0_distinct_v_0_lem}
For any divisor $D$, there exist exactly $r_0$ distinct $v_0$-reduced divisors equivalent to $D$.
}
\end{lemma}
\begin{proof}
{
First, we show that there exist at most $r_0$ distinct reduced divisors equivalent to $D$.
Suppose not, so by the pigeonhole principle, there exist two distinct
reduced divisors, $D'=D-\vec{Q}^Tf'$ and $D''=D-\vec{Q}^Tf''$
with $f'(v_0) \equiv f''(v_0) \hbox{ (mod) } r_0$. Pick $k \in \Z$
so that $(f'-f''-kR)(v_0)=0$ and let $f^*=f'-f''-kR$. By our
assumption $D' \neq D''$ and so $\vec{Q}^T(f'-f'') \neq 0$. Hence by Lemma~\ref{lem:leftkernel}, either $f^*$ or
$-f^*$ satisfies the assumptions of Lemma~\ref{firing_reduced_lemma}. Without loss of
generality, suppose $f^*$ satisfies the assumption of Lemma~\ref{firing_reduced_lemma}. But $D'=D''-\vec{Q}^Tf^*$ is a $v_0$-reduced divisor, contradicting Definition~\ref{reduced_def}(i).

Now, we show that there exist at least $r_0$ distinct reduced divisors equivalent to $D$.
Lemma~\ref{reduce_exist_lemma} implies that there exists at least one $v_0$-reduced divisor equivalent to $D$, so if $r_0 =1$ we are
done. Therefore for the rest of the proof we will assume that $r_0
> 1$. Take a $v_0$-reduced divisor $D'\sim
D$ and observe that $D''=D'-\vec{Q}^T(\chi_{\{v_0\}})$ satisfies
the condition (i) of Definition~\ref{reduced_def}. Hence Corollary~\ref{positive_divisor_reduction_cor} implies that $D''$
can be reduced without firing $v_0$ to achieve a new reduced
divisor from $D'$. We can acquire $r_0$  $v_0$-reduced divisors equivalent to $D$ by
repeated application of this method. We claim that all of the $v_0$-reduced divisors obtained are distinct. Suppose there exist $0\leq i < j <r_0$ and firing strategies $f'$ and $f''$ such that $f'(v_0)=i$, $f''(v_0)=j$, and  $D^*=D'-\vec{Q}^Tf'=D'-\vec{Q}^Tf''$ is $v_0$-reduced. This implies that $\vec{Q}^T(f''-f')={\zerov}$ but  $0<(f''-f')(v_0)<r_0$, contradicting the statement of Lemma~\ref{lem:leftkernel}.
}
\end{proof}
\begin{corollary}
{
\label{unique_firing_reduced_cor}
Let $\vec{G}$ be a directed graph and let $D$ be a divisor. There exist $r_0$ $v_0$-reduced divisors $D_i=D-\vec{Q}^Tf_i$ where $f_i(v_0)=i$ for all $0 \leq i \leq r_0-1$.
}
\end{corollary}
\begin{lemma}
{
\label{extreme_reduced_effective lem}
Let $\vec{G}$ be a directed graph and let $D$ be a divisor. Then
\begin{itemize}
\item[(i)] $D$ is equivalent to an effective divisor if and only if there exists a $v_0$-reduced divisor $D' \sim D$ such that $D'$ is effective;
\item[(ii)] Suppose $D$ is not equivalent to an effective divisor.  Then $D$ is an extreme divisor if and only if for any $v \in V(\vec{G})$, there exists a $v$-reduced divisor $D' \sim D$ such that $D'(v)=-1$.
\end{itemize}
}
\end{lemma}
\begin{proof}
$(i)$: One direction is obvious. So assume $D$ is equivalent to an effective divisor, call it $D''$. If $D''$ is $v_0$-reduced then we are done.  Otherwise, Corollary~\ref{positive_divisor_reduction_cor} implies that there exists a valid firing strategy $f$ with respect to $v_0$ such that $D''-\vec{Q}^Tf$ is $v_0$-reduced. Since $D''$ is effective and $f$ is valid with respect to $v_0$, $D''-\vec{Q}^Tf$ is effective.

$(ii)$: First assume that $D$ is an extreme divisor. The assertion of part (i) implies that for all $v \in V(D)$, if $D' \sim D$ is a $v$-reduced divisor, $D'(v) \leq -1$. Suppose there exists $v \in V(\vec{G})$ such that for all $v$-reduced divisor $D' \sim D$ we have that $D'(v) <-1$. Then by Remark~\ref{reduced_rem}, for all $v$-reduced divisors $D' \sim D$, $D'+\chi{\{v\}}$ is not effective and it is $v$-reduced. So by part (i), $D+\chi_{\{v\}}$ is not effective, a contradiction.

For proving the other direction, it is enough to show that for all $v \in V(\vec{G})$, $D+\chi_{\{v\}}$ is equivalent to an effective divisor. So let $v$ be a vertex and let $D' \sim D$ be the $v$-reduced divisor such that $D'(v)=-1$. Then $D'+\chi_{\{v\}}$ is effective and so $D+\chi_{\{v\}}$ is also.
\end{proof}
\subsubsection{Dhar's Algorithm \label{subsec:Dhar_alg}}
Dhar~\cite{Dhar90}, while studying the sand pile model, found a simple algorithm for checking whether a given divisor in an undirected graph $G$ is $v_0$-reduced or not. We discuss the directed sandpile model in the next section.  Here we generalize his algorithm so that it applies to an arbitrary directed graph $\vec{G}$. The authors found this generalization independently from Speer~\cite{Spe93}.

The input of the algorithm is a divisor $D$ satisfying the condition (i) of Definition~\ref{reduced_def}. The output of the algorithm is a finite sequence $f_i$ of firing strategies which is decreasing with respect to the $\leq$ relation. The description of the algorithm is as follows.

We construct a sequence of firing strategies $f_i$'s recursively. Set $f_0=R$. For $t \geq 0$, if there exists some $v\in V(\vec{G})\setminus \{v_0\}$ such that
\begin{equation}
\label{dharmain_eq}
(D-\vec{Q}^Tf_t)(v) \leq  -1,
\end{equation}
pick one such vertex $v$ and set
$f_{t+1}=f_t -\chi_{\{v\}}$. If for all $v \in
V(\vec{G})\setminus \{v_0\},  (D-\vec{Q}^Tf_t)(v) \geq 0$ and $f_t(v_0)>0$,
set $f_{t+1}=f_t-\chi_{\{v_0\}}$.  Otherwise the algorithm terminates and the output of the algorithm is the decreasing sequence of $f_i$'s.

We call the above algorithm the {\it generalized Dhar's Algorithm}.

\begin{theorem}
{
\label{dhar_alg_thm}
Let $D$ be a divisor satisfying condition (i) in Definition~\ref{reduced_def}. Then
\begin{enumerate}
\item[(i)]  the divisor $D$ is $v_0$-reduced if and only if the generalized Dhar's Algorithm terminates at $f_{{\onev} \cdot R}={\zerov}$.
\item[(ii)] if $D$ is a $v_0$-reduced divisor then for each $0 \leq t \leq {\onev} \cdot R -1$ such that $f_{t+1}=f_t-\chi _{\{v_0\}}$,
$D-\vec{Q}^Tf_t$ is a $v_0$-reduced divisor.
\end{enumerate}
}
\end{theorem}
\begin{proof}
{
$(i)$:  Clearly if $D$ is reduced then the algorithm terminates at
$f_{\onev \cdot R}=0$.

So assume that the algorithm terminates on the divisor $D$.
Take a valid firing $f$ with respect to $v_0$ and pick $t$ as large as possible such that $f_t \geq f$. The choice of $t$ implies that $f_{t+1}= f_t -\chi _{\{v\}}$ for some vertex $v \in V(\vec{G})\setminus\{v_0\}$ since $f(v_0)=0$. Therefore $f_t=f+f'$ where $f' \geq 0$ and $f'(v)=0$. Hence $(D-\vec{Q}^Tf)(v) = (D-\vec{Q}^Tf_t - \vec{Q}^Tf' )(v)\leq (D-\vec{Q}^Tf_t)(v)<0$ so the divisor $D$ satisfies the second condition of Definition~\ref{reduced_def}. Hence $D$ is $v_0$-reduced.

$(ii)$:  For the sake of contradiction, let $t$ be such that $f_{t+1}=f_t-\chi _{\{v_0\}}$ and
$D-\vec{Q}^Tf_t$ is not a $v_0$-reduced divisor.  There exists a valid firing strategy $f$ with respect to $v_0$
such that $((D-\vec{Q}^Tf_t)-\vec{Q}^Tf)(v) \geq 0$ for all $v \in V(\vec{G}) \setminus \{v_0\}$. Let $f'=f_t +f$, then we have two cases.
Assume there exists $v_i \in V(\vec{G})\setminus \{v_0\}$ such that $f'(v_i)>r_i$ then $f''=f'-R$ is a firing strategy which satisfies the conditions of Lemma~\ref{firing_reduced_lemma}, contradicting
the fact that for all $v \in V(\vec{G}) \setminus \{v_0\}$,
$(D-\vec{Q}^Tf')(v)>0$. Therefore, we can choose $s$ as large as possible such that $f_s \geq f'$. The choice of $s$ implies that there exists $v \in V(\vec{G})$ such that $f_s(v)=f'(v)$ and $f_{s+1}=f_s-\chi_{\{v\}}$. If $v=v_0$, since $t>s$, $f_{s+1} \geq f_t$ but $f_{s+1}(v_0) < f_t$, a contradiction. Hence $v
\in V(\vec{G}) \setminus \{v_0\}$ and $(D-\vec{Q}^Tf_{s})(v) < 0$. But $(D-\vec{Q}^Tf')(v) \leq (D-\vec{Q}^Tf_s)(v)<0$ and this contradicts the choice of $f$ and $f_t$.
}
\end{proof}

We conclude this section with the following definition which will appear in each of the subsequent sections.

\begin{definition}
\label{def:natural_RR}
Let $\vec{G}$ be a directed graph with the Riemann-Roch property.  Then
$\vec{G}$ has the natural Riemann-Roch property if its canonical
divisor $K$ has $i$th entry $\deg^+(v_i)-2$ for $0\leq i \leq n$.
\end{definition}
\subsubsection{The Sandpile Model \label{subsec:sandpile}}
The sandpile model for a directed graph is a constrained version of the ``row" chip-firing game.  We define a divisor $D$ to be a {\it $v_0$-sandpile configuration} if $D$ satisfies the condition (i) from Definition~\ref{reduced_def}. The vertex $v_0$ does not participate in this game and a vertex $v \in V(\vec{G}) \setminus \{v_0\}$ may only fire if it has at least as many chips as its out-degree (so that $v$ does not go in debt), and it never borrows. Morevover, we say that two configurations are the same if they agree at all vertices other than $v_0$.  This model has been studied in~\cite{HLMPPW08,  Lev11, Spe93}.  The goal of this section is show a connection between the sandpile model and the Riemann-Roch property for the row chip-firing game on a strongly connected directed graph.  To do this we will first show a connection between this model and $v_0$-reduced divisors.  We begin with some necessary definitions.


We now restrict our attention to the sandpile model.  We call a $v_0$-sandpile configuration {\it $v_0$-stable} if no vertex $v \in V(\vec{G}) \setminus \{v_0\}$ can fire. We note that while some authors require $v_0$ to be a global sink (in order to guarantee that a divisor will eventually stabilize), we simply insist that $v_0$ never fires. We say that a $v_0$-sandpile configuration $D'$ {\it stabilizes} to $D$, a $v_0$-stable configuration, if $D$ is $v_0$-sandpile achievable from $D'$.  To see that any $v_0$-sandpile configuration will eventually stabilize to a $v_0$-stable configuration, one may follow  an argument similar to the one from Lemma~\ref{reduce_exist_lemma}.  We note that, as the language suggests, $D$ is unique, i.e., stabilization is independent of the choice of firings, and a simple proof by induction on $k$, the length of the sequence of firings, gives this fact.   A $v_0$-stable configuration $D$ is said to be {\it $v_0$-reachable} from another $v_0$-sandpile configuration $D'$ if there exists an effective divisor $E$ such that $D'+E$ stabilizes to $D$.  A $v_0$-stable configuration is {\it $v_0$-recurrent} if it is $v_0$-reachable from any other $v_0$-sandpile configuration.

\begin{lemma}
\label{lem:recurrentdom}
A divisor $D$ is $v_0$-recurrent  if and only if there exists a divisor $D'$ such that $D'(v)\geq \deg^+(v)$ for all $v \in V(\vec{G}) \setminus \{v_0\}$ and $D'$ stabilizes to $D$.
\end{lemma}
\begin{proof}
We begin with the easier of the two directions.  Assume that $D$ is $v_0$-recurrent and let $D''$ be some divisor such that $D''(v)\geq \deg^+(v)$.  By definition, $D$ is $v-0$ reachable from $D''$, therefore there exists some effective divisor $E$ such that $D''+E=D'$ stabilized to $D$.  This gives the existence of the $D'$ in the stament of the theorem.

Conversely, given some $v_0$-sandpile configuration $D`$ such that $D'(v)\geq \deg^+(v)$ for all $v \in V(\vec{G}) \setminus \{v_0\}$, which stabelizes to $D$, we will show that $D$ is $v_0$-recurrent.  Take some $D''$, a $v_0$-sandpile configuration.  We will show that $D$ is $v_0$-reachable from $D''$.  First let $D''$ stabilize to the configuration $D'''$.  Now $D'''\leq D'$ so that $D$ is $v_0$-reachable from $D'''$.  Let $D'-D'''=E\geq 0$.  We claim that $D'' +E$ stabilizes to $D$.  By the observation made above, that stabilization is independent of a choice of firings, it is sufficient to show that there exists a sequence of firings which brings $D'' +E$ to $D$.  Because $D'' +E \geq D''$ we can perform the sequence of firings which brought $D''$ to $D'''$.  This sequence of firings brings $D'' +E$ to $D'''+E=D'$ and this now stabilizes to $D$.


\end{proof}
The following definition is for the unconstrained row chip-firing game introduced in the previous section.  We say that a divisor $D$ is $v_0$-{\it negatively achievable} from $D'$ if there exists a sequence of borrowings by individual vertices such that at each step the vertex which borrows has a negative number of chips prior to borrowing.

\begin{lemma}
\label{lem:red_borrow_neg}
A divisor $\nu$ is $v_0$-reduced if and only if there exists a divisor $D$ with $D(v) < 0$ for all $v \in V(\vec{G}) \setminus \{v_0\}$ such that $\nu$ is $v_0$-negatively achievable from $D$.
 \end{lemma}
\begin{proof}

We will first show that if $\nu$, a $v_0$-sandpile divisor, is $v_0$-negatively achievable from $D$ with $D(v) < 0$ for all $v \in V(\vec{G}) \setminus \{v_0\}$ then $\nu$ is $v_0$-reduced.  We now introduce some notation, which will be useful for this proof.  Let $S:v_{a_1}, \dots, v_{a_k}$ be the sequence of vertices which borrow and let $f_{S} \leq 0$ be the corresponding firing so that $D-Q^Tf_{S}=\nu$. Let $f_{S,j}$ be the firing strategy defined as $f_{S,j}(v)=|\{i:v_{a_i}=v, i\leq j \}|$ for $1 \leq j \leq k$, with $f_{S,0}=\vec{0}$.  Assume that $\nu$ is not $v_0$-reduced and let $f\neq \vec{0}$ be a natural firing such that $\nu-Q^Tf=\nu'$ is a $v_0$-sandpile divisor.  If $f+ f_{S}\nleq 0$ then there exists a maximal connected subset $A$ of $V(\vec{G}) \setminus \{v_0\}$ such that $(f+ f_{S})(v)>0$ for all $v\in A$, but the set $A$ loses a net positive amount of money via the firing $(f+ f_{S})$ contradicting the fact that $D-Q^T(f+ f_{S})=\nu'$ is a $v_0$ sandpile configuration and $D(v)<0$ for all $v \in A$.  Because $f+ f_{S}\leq 0$ we may take $j$ maximum so that $f_{S,j}\geq f+ f_{S}$ but $f_{S,j+1}\ngeq f+ f_{S}$.  This shows that $0\leq \nu'(v_{a_{j+1}})=(D-Q^T( f+ f_{S}))(v_{a_{j+1}})\leq (D-Q^Tf_{S,j})(v_{a_{j+1}})<0$, a contradiction.

We now show that for any $v_0$-reduced divisor $\nu$ there exists some $D$ with $D(v) < 0$ for all $v \in V(\vec{G}) \setminus \{v_0\}$ such that $\nu$ is $v_0$-negatively achievable from $D$. Take $\nu$ and greedily fire vertices in $v \in V(\vec{G}) \setminus \{v_0\}$ with an nonnegative number of chips until you obtain $D$ with $D(v) < 0$ for all $v \in V(\vec{G}) \setminus \{v_0\}$.  To see that this process will eventually terminate adapt the argument give in Lemma~\ref{reduce_exist_lemma} for why greedy reduction of a divisor terminates.  We claim that $D$ is the desire divisor.  If we now, as above, we greedily borrow by vertices in $v \in V(\vec{G}) \setminus \{v_0\}$ which are in debt, we will stop at a $v_0$-reduced divisor $\nu'$.  To see that this process eventually terminates, again mimic the argument from Lemma~\ref{reduce_exist_lemma}.  The fact that $\nu'$ is $v_0$-reduced was proven above.  The divisor $\nu'$ is clearly equivalent to $\nu$, and $v_0$ did not participate in the above process, hence the divisor obtained is equal to $\nu$.
\end{proof}

The authors, independently from Speer~\cite{Spe93}, discovered the following theorem.

\begin{theorem}
\label{thm:recurrent_reduced}
A $v_0$-sandpile configuration $D$ is $v_0$-recurrent if and only if the divisor $\nu$ is a $v_0$-reduced divisor, where $\nu(v_i)=\deg^+(v_i)-1-D(v_i)$ for all $0 \leq i \leq n$.
\end{theorem}
\begin{proof}
Let $K$ be the divisor such that $K(v_i)=deg^+(v_i)-2$.  We first note that the map $\phi (D)= K+\vec{1}-D$ is a bijection between divisors $D$ such that $D(v)\geq \deg^+(v)$  for all $v \in V(\vec{G}) \setminus \{v_0\}$ and divisors $D$ such that $D(v)<0$  for all $v \in V(\vec{G}) \setminus \{v_0\}$.  The theorem then follows by observing that $\nu$ is $v_0$-negatively achievable from $D$ with $D(v) <0$  for all $v \in V(\vec{G}) \setminus \{v_0\}$ if and only if $\phi(\nu)$ is $v_0$-sandpile achievable from $\phi(D)$ with $(\phi(D))_i\geq \deg^+(v_i)$ for all $v \in V(\vec{G}) \setminus \{v_0\}$.
\end{proof}
We note that using the notion of equivalence given by the unconstrained row chip-firing game, the previous theorem shows that there are exactly $r_0$ $v_0$-recurrent divisors in each equivalence class.  This is different from the case of undirected graphs or directed graphs with $v_0$ a global sink, where the recurrent state in each equivalence class is unique.

We define a divisor $D$ to be minimally $v_0$-recurrent if, ignoring the value of $D(v_0)$, it is minimal with respect to dominance among all $v_0$-recurrent divisors.  Using this definition we have a new way of describing the natural Riemann-Roch property in terms of the sandpile model for strongly connected directed graphs.
\begin{theorem}
A directed graph, $\vec{G}$ has the natural Riemann-Roch property if and only if for each minimal $v_0$-recurrent divisor $D$ there exists $D'=D+ke_0$, $k\in \Z$, $E_i \in \Z_{\geq 0}$ for $0 \leq i \leq n$ such that $E_i(v_i)=0$ and $E_i(v_j)>0$ for $j \neq i$ and $D' \sim E_i$ and each $D'$ is of fixed degree $g-1\in \N$.
\end{theorem}
\begin{proof}
Clearly $D$ is minimally $v_0$ recurrent if and only if, by Theorem \ref{thm:recurrent_reduced}, we may fix $D'$ as in the statement of the theorem such that $\nu=K-D' +\onev$  is extreme $v_0$-reduced.  Hence, $\vec{G}$ has the natural Riemann Roch property if and only if $\nu'= D' - \onev \in Ext(\Sigma(\L)$ and is fixed degree $g-1$, which occurs precisely when $D' \in Ext(\Sigma_{\R}(\L))$ and is of fixed degree $g-1$. By Lemma \ref{extreme_sigma_closure_thm}, the Theorem follows.
\end{proof}
\subsection{Column Chip-Firing Game, $\vec{G}$-Parking Functions, and Riemann-Roch Theory \label{subsec:G-parking}}
In this section we present a chip-firing game which comes from the columns of the Laplacian matrix.

\begin{definition}
{
\label{g-parking_def} We call a divisor $D$ a directed $\vec{G}$-parking function (or simply $\vec{G}$-parking) with respect to $v_0$ if the following two conditions hold:
\begin{enumerate}
\item[(i)] for all $v \in V(\vec{G})\setminus  \{v_0\}, D(v)\geq 0$,
\item[(ii)] for every set $A \subseteq V(\vec{G}) \setminus \{v_0\}$, there exists some $v \in A$ such that $|{\{(v,u) \in E(\vec{G}): u \notin A}\}|\geq D(v)$.
\end{enumerate}
}
\end{definition}
We introduce the following ``column" chip-firing game wherein if a vertex $v$ {\it fires}, it loses $\deg^+(v)$ chips and sends a chip along each incoming edge $(u,v) \in E(\vec{G})$ ({\it borrowing} is defined as the inverse of firing).  Note that the total number of chips is not preserved by firing in contrast to the previous ``row" chip-firing game. It is not hard to see that if all vertices in a set $A$ fire once then a vertex $v \in A$ will lose as many chips as it has edges leaving $A$, i.e., $|{\{(v,u): u \notin A}\}|$, while a vertex $u \not \in A$ will gain as many chips as it has edges entering to it from $A$, i.e., $|{\{(v,u): v \in A}\}|$.  One may view this game as a walk through the lattice spanned by the columns of the Laplacian of $\vec{G}$ and it follows immediately that if $D$ is a divisor then $(D-\vec{Q}\chi_{A})(v)=D(v)-|{\{(v,u): u \notin A}\}|$ if $v \in A$ and $(D-\vec{Q}\chi_{A})(u)=D(u)+|{\{(v,u): v \in A}\}|$ if $u \notin A$. Because $\vec{Q}$ is orthogonal to $\onev$, i.e., $\vec{Q}\onev=\zerov$, we have that for any firing strategy $f$ there exists some firing strategy $f'$ such that $\vec{Q}(f-f')=\zerov$ and $f'\leq \chi_A$ for some $A \subseteq V(\vec{G}) \setminus \{ v_0 \} $.  It is also worth mentioning that if $R=(r_0, \dots, r_n) \in \N^{n+1}$ is the vector guaranteed by Lemma~\ref{lem:leftkernel} such that $R^T\vec{Q}={\zerov}^T$, then $\deg_R(\vec{Q}f)=0$ for all $f \in \Z^{n+1}$, i.e., the total number of chips is preserved in the ``column" chip-firing game with respect to $\deg_R(\cdot)$. One may also interpret this fact combinatorially by assigning to each vertex $v_i$ its own ``chip currency'' worth $r_i$ of a ``universal chip currency''.  Similar notions of ``currencies'' and ``exchange rates'' are employed when discussing chip-firing on arithmetical graphs in Section~\ref{AGrpahs}.

A $\vec{G}$-parking function with respect to $v_0$ is a divisor $D$ such that $D(v)\geq 0$ for all $v\in V(\vec{G})\setminus \{v_0\}$ and for each set $A\subset V(\vec{G})\setminus \{v_0\}$ there exists some vertex $v\in A$ such that $|{\{(v,u): u \notin A}\}|>D(v)$.  This definition is precisely analogous to the definition of a $v_0$-reduced divisors from the ``row" chip-firing game. More specifically, if we change $\vec{Q}^T$ to $\vec{Q}$ in definition of $v_0$-reduced divisor (Definition~\ref{reduced_def}), then we get the definition of $\vec{G}$-parking function with respect to $v_0$ (Definition~\ref{g-parking_def}). Hence, Dhar's algorithm introduced in~\cite{BN07, Dhar90} applies in verifying whether $D$ is $\vec{G}$-parking function with respect to $v_0$. Note that for undirected graphs the notion of a $v_0$-reduced divisor and a $G$-parking function agree as the Laplacian is symmetric, i.e., the ``row" and ``column" chip-firing games are identical.  It is a well known fact, and has several combinatorial proofs, that the $\vec{G}$-parking functions are in bijection with set of rooted directed spanning trees~\cite{CP05}.

An {\it Eulerian} directed graph $\vec{H}$ is a directed graph such that $\deg^+(v)=\deg^-(v)$ for each $v \in V(\vec{H})$.  The name is derived from the fact that they are exactly those directed graphs which possess a directed Eulerian circuit. 
\begin{theorem}
\label{thm:RR_column_chip}
Let $\vec{G}$ be a strongly connected directed graph with Laplacian $\vec{Q}$ and let $\vec{G'}$ be the Eulerian directed graph with Laplacian $\vec{Q}^T\mathcal{R}$ where $\mathcal{R}=diag(r_0, \dots, r_n)$ where$ \mathcal{R}\vec{1}\vec{Q}$.  The directed graph  $\vec{G}$ has the Riemann-Roch property for the column chip-firing game if and only if the directed graph $\vec{G'}$ has the Riemann-Roch property for the row chip-firing game.
\end{theorem}
\begin{proof}
Let $\L'_{\vec{G}}=\{\vec{Q}f: f \in \Z^{n+1}\}$ be the lattice spanned by the columns of $\vec{Q}$.  It follows by Theorem~\ref{thm:RR_R_one} that $\L'_{\vec{G}}$ has the Riemann-Roch property if and only if $\mathcal{R}\L'_{\vec{G}}$ does.  This is the lattice spanned by the rows of $\vec{Q}^T\mathcal{R}$ completing the proof.
\end{proof}

We note that the column chip-firing game for an Eulerian digraph is the same game as the row chip-firing game played on the same directed graph with as of the orientations of all of the arrows reversed.  This explains why we are passing to the transpose of the Laplacian in the proof.

Amini and Manjunath~\cite{AM09} have some results related to Eulerian directed graphs (which they call regular digraphs).  By the previous theorem, all of these results extend to the column chip-firing game on strongly connected directed graphs.  We also remark that for testing whether a divisor is $v_0$-reduced, the burning algorithm of Dhar may be applied (burning along incoming edges) and this algorithm can be used to obtain several of the results of Amini and Manjunath related to Eulerian directed graphs.
\section{Arithmetical Graphs \label{AGrpahs}}
Let $G$ be a connected undirected multigraph, choose an ordering $\{v_0, \dots, v_n\}$ of vertices of $G$, and let $A$ be the corresponding {\it adjacency matrix} of $G$. Let $R=(r_0, \dots, r_n)^T \in \Bbb N^{n+1}$ be such that $gcd(r_0,r_1 \dots, r_n)=1$ and let  $\delta_0, \dots, \delta_{n} \in \Bbb N$ be such that $(\mathcal{D}-A)R=\zerov$, where $\mathcal{D}=diag(\delta_0, \dots, \delta_{n})$.  We say $(G,R)$ is an {\it arithmetical graph} with Laplacian $Q=\mathcal{D}-A$ and corresponding multiplicity vector $R$, where for all $0 \leq i \leq n$ the value $r_i$ is the {\it multiplicity} of the vertex $v_i$. Note that an undirected graph $G$ can be considered as an arithmetical graph $(G,{\onev})$.

Consider the following chip-firing game played on the vertices of an arithmetical graph $(G,R)$. Suppose we have a ``universal chip currency'' and each vertex $v_i$ has its own ``$v_i$-chip currency'' such that each $v_i$-chip is worth $r_i$ of the ``universal chip currency''.  If a vertex $v_i$ {\it fires}, it loses $\delta_i$ of its own $v_i$-chips and sends $m_{i,j}$ $v_j$-chips to each $v_j$ adjacent to $v_i$, where $m_{i,j}$ is the number of edges between $v_i$ and $v_j$.  We define {\it borrowing} to be the inverse of firing.  Let $\L_{(G,R)}$ be the lattice spanned by the columns of $Q$.  It is easy to see that moves in this chip-firing game correspond to translations of some divisor $D$ by a lattice point $l \in \L_{(G,R)}$.  This observation allows us to make use of definitions and theorems from Section 2 when discussing the chip-firing game.

Let $(G,R)$ be an arithmetical graph and $\mathcal{R}=diag(r_0, \dots, r_{n})$. Let $\vec{G}_R$ be the directed graph obtained from $(G,R)$ by replacing each undirected edge $(v_i,v_j)$ with $r_j$ edges directed from $v_i$ to $v_j$ and $r_i$ edges directed from $v_j$ to $v_i$.  The chip-firing game for $(G,R)$ corresponds to the row chip-firing game for $\vec{G}_R$ by converting each vertex's currency to the universal chip currency.   If we define $\vec{Q}_R$ be the Laplacian of $\vec{G}_R$ we may observe that $\vec{Q}_R^T =\mathcal{R}Q$.   By Theorem~\ref{thm:RR_R_one}, it follows that the chip-firing game on $(G,R)$ will have the Riemann-Roch property if and only if the row chip-firing game on $\vec{G}_R$ has the Riemann-Roch property. The row chip-firing game on $\vec{G}_R$ is strictly ``finer'' than the chip-firing game on $(G,R)$ in the sense that a vertex, $v_i$ need not have a multiple of $r_i$ universal chips, although by the previous observation this difference does not effect whether the Riemann-Roch property holds.

In our discussion of the chip-firing game for arithmetical graphs we will borrow several definitions and methods from the row chip-firing game whose interpretation will be clear from the context in which they are used.  In particular the definition of a $v_0$-reduced divisor and the generalized Dhar's algorithm will be frequently employed.

\begin{theorem}
\label{asdf}
Let $(G,R)$ be an arithmetical graph with Laplacian $Q$ and let $\vec{G}_R$ be the associated directed graph.  Then $\vec{G}_R$ has the Riemman-Roch property for the column chip-firing game.
\end{theorem}
\begin{proof}
By Theorem~\ref{thm:RR_column_chip} it is equivalent to ask the question for the row chip-firing game on the directed graph $\vec{H}$ whose Laplacian is $\mathcal{R}\vec{Q'}$ where $\vec{Q'}$ is the Laplacian for $\vec{G}_R$.  But $\vec{Q'}$ is simply $\vec{Q}\mathcal{R}$ and so $\vec{H}$ has Laplacian $\mathcal{R}\vec{Q}\mathcal{R}$ which as one can easily check is the Laplacian of the undirected graph obtained from $G$ by replacing each edge $(v_i, v_j)$ with $r_ir_j$ edges.  By Baker and Norine, this graph has the Riemman-Roch property and this completes the proof.
\end{proof}

Let $\mathcal{N}=\{D \in Ext(\Sigma(\L_{(G,R)})): \deg_R(D)=g_{\max}-1\}$. For each $0 \leq i \leq n$, let $N(v_i)$ denote the family of vertices which are adjacent to $v_i$ counting their multiplicities. We call $|N(v_i)|$ the {\it degree} of the vertex $v_i$ and we denote it by $\deg(v_i)$. Recall the definition of $g_0$, the number such that $2g_0-2=\sum_{i=0}^n r_i(\delta_i-2)$. It is not hard to verify and is noted in~\cite{Lor89} that $g_0$ is an integer. It is also easy to see that by firing all of the vertices of the $G$, we get $\sum_{i=0}^n r_i\delta_i=\sum_{i=0}^n r_i\deg(v_i)$. Therefore $2g_0-2=\sum_{i=0}^n r_i(\deg(v_i)-2)$.
 \begin{theorem}
{
\label{gmaxg_0}Let $(G,R)$ be an arithmetical graph.  Then  $g_{\max}\leq g_0$.
}
\end{theorem}
\begin{proof}
{
The following proof is an averaging argument employing the generalized Dhar's algorithms and gives a bound twice as good as the naive bound.  If one looks closely at the proof, it becomes apparent that arithmetical graphs are precisely those ``directed graphs'' for which such an averaging argument is effective.
Let $D \in \mathcal{N}$. Choose a $v_0$-reduced divisor $D' \sim D$ such that $D'(v_0)$ is as large as possible. For proving the theorem, it is enough to show that $\deg_R(D') \leq g_0-1$. Apply the generalized Dhar's algorithm to $D'$. For all $0 \leq i \leq n$ and $1 \leq k \leq r_i$, define
$\mathcal{F}_{i,k}$ to be the firing strategy obtained from the generalized Dhar algorithm such that
$\mathcal{F}_{i,k}(v_i)=k$ and the successor of $\mathcal{F}_{i,k}$ is the firing strategy
$\mathcal{F}_{i,k}-\chi_{\{v_i\}}$. For each $v_i \in V(\vec{G}) \setminus
v_0$ we obtain $r_i$ inequalities as follows:\\
for each $k$ where $1 \leq k  \leq r_i$, we have:
\begin{equation}
\label{othervertex_max-ing}
D'(v_i) \leq k\delta_i - \left(\sum_{v_j \in N(v_i)} \mathcal{F}_{i,k}(v_j)\right) -1,
\end{equation}
which follows from the fact that $(D'-Q\mathcal{F}_{i,k})(v_i) <0$ by choice of $\mathcal{F}_{i,k}$.
\\For the vertex $v_0$, we know that for all $1 \leq k \leq r_0$, $$k\delta_0 - \sum_{v_j \in N(v_0)} \mathcal{F}_{0,k}(v_j) \geq 0,$$
by the choice of $D'$ and the second assertion of Lemma~\ref{dhar_alg_thm}.  Because $D' \in \mathcal{N}$, by (ii) of Lemma~\ref{extreme_reduced_effective lem} we have that $D'(v_0)< 0$. Hence, for all $1 \leq k \leq r_0$,
\begin{equation}
\label{v_0max-inq}
  D'(v_0) \leq k\delta_0 - \left(\sum_{v_j \in N(v_0)} \mathcal{F}_{0,k}(v_j)\right)-1.
\end{equation}
Note that $\sum_{i=0}^{n} \sum_{k=1}^{r_i} D'(v_i)=D' \cdot R=\deg_R(D')$.

Now, taking the sum over all inequalities in~(\ref{othervertex_max-ing}) and~(\ref{v_0max-inq}), we have:
\begin{equation}
\label{main_ieq}
\sum_{i=0}^{n} \sum_{k=1}^{r_i} D'(v_i)\leq \sum_{i=0}^{n}r_i((r_i+1)\delta_i -2)/2 - \sum_{i=0}^{n} \sum_{k=1}^{r_i} \sum_{v_j \in N(v_i)}\mathcal{F}_{i,k}(v_j).
\end{equation}

We will now restrict our attention to $\sum_{i=0}^{n} \sum_{k=1}^{r_i} \sum_{v_j \in N(v_i)}\mathcal{F}_{i,k}(v_j)$.  By reordering the sums, we have
$$\sum_{i=0}^{n} \sum_{k=1}^{r_i} \sum_{v_j \in N(v_i)}\mathcal{F}_{i,k}(v_j)=\sum_{i<j, ~ v_iv_j \in E(G)}\left( \sum_{k=1}^{r_i}  \mathcal{F}_{i,k}(v_j) + \sum_{\l=1}^{r_j} \mathcal{F}_{j,\l}(v_i)\right).$$
We claim that if $v_iv_j \in E(G)$ then $\sum_{k=1}^{r_i}  \mathcal{F}_{i,k}(v_j) + \sum_{\l=1}^{r_j} \mathcal{F}_{j,\l}(v_i) = r_ir_j$.  We prove the claim by induction on $r_i+r_j$. If $r_i+r_j=2$, then the claim holds trivially, since $r_i=r_j=1$. Now suppose $r_i+r_j=m \geq 3$. Without loss of generality, assume $\mathcal{F}_{i,r_i}$ is generated before $\mathcal{F}_{j,r_j}$ in the run of the generalized Dhar's algorithm on $D'$. Hence
$$\sum_{k=1}^{r_i}  \mathcal{F}_{i,k}(v_j) + \sum_{\l=1}^{r_j} \mathcal{F}_{j,\l}(v_i) = r_j + \sum_{k=1}^{r_i-1}  \mathcal{F}_{i,k}(v_j) + \sum_{\l=1}^{r_j} \mathcal{F}_{j,\l}(v_i) = r_j+(r_i-1)r_j=r_ir_j.$$
The equality $\sum_{k=1}^{r_i-1}  \mathcal{F}_{i,k}(v_j) + \sum_{\l=1}^{r_j} \mathcal{F}_{j,\l}(v_i)=(r_i-1)r_j$ follows from induction hypothesis. This completes the proof of the claim.
So $$\sum_{i<j, ~ v_iv_j \in E(G)}\left( \sum_{k=1}^{r_i}  \mathcal{F}_{i,k}(v_j) + \sum_{\ell=1}^{r_j} \mathcal{F}_{j,\ell}(v_i)\right) = \sum_{i<j, ~ v_iv_j \in E(G)} r_ir_j={1 \over 2}\left(\sum _{i=0}^{n}r_i \sum_{v_j \in N(v_i)} r_j\right).$$
Since $QR=0$, for all $0 \leq i \leq n$, $\sum_{v_j \in N(v_i)} r_j=r_i\delta_i$.
Hence
\begin{equation}
\label{iductionobtained_eq}
\sum_{i=0}^{n} \sum_{k=1}^{r_i} \sum_{v_j \in N(v_i)}\mathcal{F}_{i,k}(v_j)={1 \over 2}\left(\sum _{i=0}^{n}r^2_i\delta_i \right).
\end{equation}
Now by substituting~(\ref{iductionobtained_eq}) into inequality~(\ref{main_ieq}), we have:
$$\deg_R(D')\leq \sum _{i=0}^{n}(r_i((r_i+1)\delta_i -2)/2 - {1 \over 2}\left(\sum _{i=0}^{n}r^2_i\delta_i \right) = \sum _{i=0}^{n}r_i(\delta_i -2)/2= g_0-1.$$
}
\end{proof}
So the above theorem shows that if, in a configuration of the game identified by $D \in Div((G,R))$, $\deg_R(D) \geq g_0$, then $D$ has a winning configuration.
\begin{corollary}
{
\label{g=g_0cor}
$g_{\max}=g_0$ if and only if all inequalities in~(\ref{othervertex_max-ing}) and~(\ref{v_0max-inq}) obtained in
a run of the generalized Dhar's algorithm on a $v_0$-reduced
divisor $D \in \mathcal{N}$ are tight, i.e. if $f_i$ is the sequence of firing strategies obtained from  the run of the generalized Dhar's algorithm on a $v_0$-reduced
divisor $D \in \mathcal{N}$, for all $0 \leq t \leq {\onev} \cdot R-1 $, if $f_{t+1}=f_t
- \chi_{\{v\}}$ then $(D-Q(f_t))(v)=-1$.
}
\end{corollary}

It is clear, and demonstrated below, that if $D\in \mathcal{N}$ and deg$(D)$=$g_{max}$-1, then for each $v \in V(G)$ and $D' \sim D$ such that $D'$ is $v$-reduced, we have $D'(v)=-1$.  The following theorem shows that the converse is also true.

\begin{theorem}
{
Let $D \in \mathcal{N}$. Then $\deg(D)=g_{max}-1$ if and only if for each $D' \sim D$ such that $D'$ is a $v$-reduced divisor, $D'(v)=-1$.
}
\end{theorem}
\begin{proof}
{
Suppose $D\in \mathcal{N}$ with $\deg(D)=g_{max}-1$.  Take $v \in
V(\vec{G})$. By applying (ii) of Lemma~\ref{extreme_reduced_effective lem} we may pick $D'\sim D$ to be a $v$-reduced divisor such that $D'(v)=-1$. Corollary~\ref{g=g_0cor} implies that all the inequalities are tight, so
for all $v$-reduced divisor $D''\sim D$, $D''(v)=-1$.

Conversely, assume that $D\in \mathcal{N}$ is $v_0$-reduced and suppose that for each $D' \sim D$ which is an extreme $v$-reduced divisor, $D'(v)=-1$.  We wish to show that $\deg(D)=g_{max}-1$. Apply the generalized Dhar's algorithm to $D$, and define $\mathcal{F}_{i,k}$ to be the firing strategy obtained from the generalized Dhar algorithm such that $\mathcal{F}_{i,k}(v_i)=k$ and the successor of $\mathcal{F}_{i,k}$ is the firing strategy $\mathcal{F}_{i,k}-\chi_{\{v_i\}}$.
\begin{equation}
\label{copyothervertex_max-ing}
D(v_i) \leq k\delta_i - \left(\sum_{v_j \in N(v_i)} \mathcal{F}_{i,k}(v_j)\right) -1,
\end{equation}
which follows from the fact that $(D-Q\mathcal{F}_{i,k})(v_i) <0$ by choice of $\mathcal{F}_{i,k}$.  By the previous corollary, to show that $\deg(D)=g_{max}-1$, it is enough to show that each of the inequalities from  (\ref{copyothervertex_max-ing}) hold with equality.

For the vertex $v_0$, we know that for all $1 \leq k \leq r_0$,

$$k\delta_0 - \sum_{v_j \in N(v_0)} \mathcal{F}_{0,k}(v_j) \geq 0,$$

this follows from the choice of $D$ and the second assertion of Lemma~\ref{dhar_alg_thm}.  Because $D$ is extreme, by (ii) of Lemma~\ref{extreme_reduced_effective lem} we have that $D(v_0)< 0$. Hence for all $1 \leq k \leq r_0$,
\begin{equation}
\label{copyv_0max-inq}
  D(v_0) \leq k\delta_0 - \left(\sum_{v_j \in N(v_0)} \mathcal{F}_{0,k}(v_j)\right)-1.
\end{equation}

By assumption all of the inequalities for $v_0$ above hold with equality. So
take $v_i \in V(\vec{G})\setminus v_0$ and $1\leq k \leq r_i$. For finishing the proof, we will show that
$(D- Q(\mathcal{F}_{i,k}))(v_i)=-1$. Let the firing strategy $f$ be such that
$D - Qf$ is $v_i$-reduced and $f(v_i)=k$, where the existence of $f$ is guaranteed by Corollary~\ref{unique_firing_reduced_cor}.
Assume $f' \approx f$ is a natural firing strategy. Let $f_t$'s be the sequence of firing strategies obtained from a run of the generalized Dhar's algorithm on $D$. Take $j$ as large as possible such that $f_j \geq f'$.
Let $v \in V(\vec{G})$ be such that $f_{j+1}=f_j-\chi_{\{v\}}$. Let the firing strategy $f''$ be such that $f'=f_j-f''$ where $f'' \geq \zerov$ and $f''(v)=0$. We claim that $v=v_i$.  If $v \notin \{v_0,v_i\}$  then $(D-Qf')(v)=(D'-Q(f_j-f''))(v)\leq (D-Q(f_j))(v)<0$,
contradicting the fact that $D-Qf'$ is a $v_i$-reduced. If $v=v_0$, then $(D-Qf')(v_0)=(D-Q(f_j-f''))(v_0) \leq (D-Q(f_j))(v_0)=-1$ since  $D-Qf_j$ is a  $v_0$-reduced divisor by the second part of Theorem~\ref{dhar_alg_thm}. But this again contradicts the
fact that $D-Qf'$ is a $v_i$-reduced divisor. Hence $v=v_i$ and this finishes the proof of the claim. Therefore $f_j=\mathcal{F}_{i,k}$ and we have: $$-1=(D-Qf')(v_i)=(D-Q(f_j-f''))(v_i)=(D-Q(\mathcal{F}_{i,k}-f''))(v_i)\leq
(D-Q(\mathcal{F}_{v_i,k}))(v_i) \leq-1.$$ Hence $(D-Q(\mathcal{F}_{i,k}))(v_i)=-1$ as desired.
}
\end{proof}
We note that a more general version of the previous theorem can be stated for strongly connected directed graphs and might have been included in the section on Dhar's algorithm, but because we do not have statement like Corollary \ref{g=g_0cor} for all strongly connected directed graphs, the statement of this more general theorem would have been awkwardly phrased.

\begin{theorem}
{
\label{g=g_0canonical_thm}
Let $K = (\delta_0-2, ... , \delta_n-2)$ be a
vector in $\Z^{n+1}$. If $g_{\max}=g_0$ then $D \in \mathcal{N}$ if and only if $K-D \in \mathcal{N}$.
}
\end{theorem}
\begin{proof}
{
Without loss of generality, we may assume $D$ is a $v_0$-reduced divisor. Apply the generalized
Dhar's algorithm on $D$ and let $f_i$ be the output sequence. Let $\mathcal{F}_{i,k}$ be the firing strategies defined in the proof of Theorem~\ref{gmaxg_0}.

Define the divisor $D'$ such that for all $0 \leq i \leq n$, $$D'(v_i)= k\delta_i-\left(\sum_{v_j \in
N(v_i)} (R-\mathcal{F}_{i,r_i+1-k})(v_j)\right)-1.$$
We claim that $D'$ is well-defined. For proving the claim, it is enough to show that for all $0 \leq i \leq n$, the value of $D'(v_i)$ does not depend upon $k$. We will show $D'=K-D$.

Since $g_{\max}=g_0$, Corollary~\ref{g=g_0cor} implies that for all $0 \leq i \leq n$, $\sum_{v_j \in
N(v_i)} \mathcal{F}_{i,r_i+1-k}(v_j)=(r_i+1-k)\delta_i-D(v_i)-1$. For all $0 \leq i \leq n$, we have:
$$\sum_{v_j \in
N(v_i)} (R-\mathcal{F}_{i,r_i+1-k})(v_j)=\left(\sum_{v_j \in
N(v_i)}r_j\right)-\left((r_i+1-k)\delta_i-D(v_i)-1\right)=-\delta_i+k\delta_i+D(v_i)+1.$$
Therefore,
$$D'(v_i)= k\delta_i-\left(\sum_{v_j \in
N(v_i)} (R-\mathcal{F}_{i,r_i+1-k})(v_j)\right)-1=k\delta_i-(-\delta_i+k\delta_i+D(v_i)+1)-1=\delta_i-2-D(v_i).$$
Since $\deg_R(K-D)=g_0-1$, for finishing the proof we only need to show that $K-D$ is not equivalent to an effective divisor.

Assume to the contrary that $D'$ is equivalent to some effective divisor $E$ and let $f$ be such that $D'-Qf=E$. Let $f' \approx
 f$ be a natural firing strategy guaranteed by Lemma~\ref{natural_lemma}.
Define a ``reverse sequence'' of firing strategies $f'_i=R-f_{\onev \cdot R-i}$ for all $0 \leq i \leq \onev \cdot R$. Take $t$ as large as possible such that $f'_t \geq f'$. So there exists $v_i \in V(\vec{G})$ such that $f'(v_i)=f'_t(v_i)$. By the definition of the reverse sequence, there exists $1 \leq k \leq r_i$ such that $f'_t=R-\mathcal{F}_{i,r_i+1-k}+\chi_{\{v_i\}}$. Therefore,
$$E(v_i) \leq (D'-Qf'_t)(v_i)$$$$=k\delta_i-\left(\sum_{v_j \in
N(v_i)} (R-\mathcal{F}_{i,r_i+1-k})(v_j)\right)-1-\left(r_i-(r_i+1-k)-1\right)\delta_i+\left(\sum_{v_j \in
N(v_i)} (R-\mathcal{F}_{i,r_i+1-k} + \chi_{\{v_i\}})(v_j)\right)$$
$$=k\delta_i-\left(r_i-(r_i+1-k)-1\right)-1=-1.$$
Note that $\sum_{v_j \in
N(v_i)} (R-\mathcal{F}_{i,r_i+1-k} + \chi_{\{v_i\}})(v_j)=\sum_{v_j \in
N(v_i)} (R-\mathcal{F}_{i,r_i+1-k})(v_j)$.
This contradicts the choice of $E$. Hence $D'=K-D$ is not equivalent to an effective divisor.
}
\end{proof}
We should mention that Theorem~\ref{gmaxg_0} and Theorem~\ref{g=g_0canonical_thm} are due to Lorenzini~\cite{Lor09}. His approach in proving these theorems is purely algebraic. As mentioned in~\cite{Lor09}, he was interested in combinatorial proof of these facts which could be the one presented in this paper.

\begin{theorem}
{
\label{g_0g_ming_max_equal}Let $(G,R)$ be an arithmetical graph. If $g_0=g_{\min}=g_{\max}$, then $(G,R)$ has the Riemann-Roch property.  Moreover, the corresponding directed graph has the natural Riemann-Roch property.
}
\end{theorem}
\begin{proof}
{
The first part of the theorem follows as an immediate consequence of Theorem~\ref{RR_formula_equiv_U_RI_thm} and Theorem~\ref{g=g_0canonical_thm}.  The second part of the theorem follows by Corollary \ref{cor:canonical_1_R}, which in this context says that if $g_0=g_{\min}=g_{\max}$, then the canonical divisor for the corresponding digraph $\vec{G_R}$ has $i$th entry $\deg^+(v_i)-2$, i.e., $\vec{G_R}$ satisfies Definition \ref{def:natural_RR} for the row chip-firing game.  Moreover, we note that $(\delta_0-2, ... , \delta_n-2) \sim (\deg(v_0)-2, ..., \deg(v_n)-2)$ as is easily observed by computing $Q \vec 1$.
}
\end{proof}

\begin{corollary}
{
\label{unique_extreme} Let $(G,R)$ be an arithmetical graph. If $\L_{(G,R)}$ has a unique class of extreme divisors, i.e. $Ext(\Sigma(\L_{(G,R)}))=\{\nu+\l :\l \in \L_{(G,R)}\}$, then $\L_{(G,R)}$ has the Riemann-Roch property.
}
\end{corollary}
\subsection{Arithmetical Graphs with the Riemann-Roch Property}
\begin{theorem}
{
\label{g_0_less_than_one} Let $(G,R)$ be an arithmetical graph. If $g_0\leq 1$ then $(G,R)$ has the Riemann-Roch property.
}
\end{theorem}
\begin{proof}
{
Let $v_0$ be a vertex such that $r_0 \leq r_i$ for all $1 \leq i \leq n$. Let $D$ be an extreme $v_0$-reduced divisor with $D(v_0)=-1$. By Theorem~\ref{gmaxg_0} $g_{\max} \leq g_0$, so $\deg(D) \leq g_{\max}-1 \leq 0$. Now we have two cases:
\begin{itemize}
\item[(i)] $D(v_i)=0$ for all $1 \leq i \leq n$, part (ii) of Lemma~\ref{extreme_reduced_effective lem} and the choice of $r_0$ implies that $D$ is the unique extreme $v_0$-reduced divisor, and the assertion of the lemma holds by Corollary~\ref{unique_extreme}. Note that in this case $g_{\max} \neq g_0$ unless $g_0=0$ and $r_0=1$.

\item[(ii)] There exists $1 \leq i \leq n$ such that $D(v_i) > 0$. Since $\deg(D) \leq 0$, $r_i=r_0$ and $v_i$ is the only vertex with $D(v_i)>0$. This implies that the divisor $D'$ with $D'(v_0)=-1$ and $D'(v_j)=0$ for all $1 \leq j \leq n$ is not an extreme divisor. Hence, $g_0=g_{\min}=g_{\max}=1$, and assertion of the lemma follows by Theorem~\ref{g_0g_ming_max_equal}.
	
\end{itemize}
}
\end{proof}
Using the definition of $g_0$ the following is immediate consequence of the Theorem~\ref{g_0_less_than_one}.
\begin{corollary}
{
Let $(G,R)$ be an arithmetical graph with all $\delta_i$'s equal to two or all $\deg(v_i)$'s equal to two. Then $(G,R)$ has the Riemann-Roch property.
}
\end{corollary}
The former arithmetical graphs are those coming from the connection between Lie algebras or elliptical curves which have been classified~\cite{CSM95} and the latter arithmetical graphs where the underlying graph is a cycle.
The following two examples show that both cases described in the proof of Theorem~\ref{g_0_less_than_one} occur.
\begin{example}
{
Let $(G,R)$ be an arithmetical graph where $G$ is the even cycle $v_0, \dots, v_{2n-1}$ for $n \geq 2$, and for all $0 \leq i \leq n-1$, the multiplicities of the vertices $v_{2i}$ and $v_{2i+1}$ are $1$ and $2$, respectively. Then $g_{\min}=g_{\max}=g_0=1$, and in particular $(G,R)$ has the Riemann-Roch property.
}
\end{example}
\begin{proof}
{
We claim that the set of extreme $v_0$-reduced divisors for $(G,R)$ are the set of divisors $D_i=\chi_{\{v_{2i}\}}
-\chi_{\{{v_0}\}}$ for all $1\leq i \leq n-1$.
Assume $1 \leq i \leq n-1$, and the vector $f$ is a valid firing strategy with respect to $v_0$ such that $D_i-Qf \geq {\zerov}$.
Observe that if $f(v_{2i})=1$, then in order to $(D_i-Qf)(v_{2i}) \geq 0$ we must have $f(v_{2i-1})+f(v_{2i-1}) \geq 3$. By symmetry, assume that $f(v_{2i-1}) \geq 2$. Since $(D_i-Qf)(v_{2i-1}) \geq 0$, we have $f(v_{2i-2}) =1$. By repeating the argument, we conclude that $f(v_0)=1$, a contradiction. This shows that $D_i$ is $v_0$-reduced and since $r_0=1$, (i) of Lemma~\ref{extreme_reduced_effective lem} implies that $D_i$ is not equivalent to an effective divisor. For proving the fact that $D_i$ is an extreme divisor, it is enough to show that
$D_i+\chi_{\{v_j\}}$ is equivalent to an effective divisor, for all $0\leq j \leq 2n-1$.

It is easy to see that $g_0=1$. If $0 \leq j \leq 2n-1$ is odd, then the divisor $D_i+\chi_{\{v_j\}}$ has degree $2 > g_0$, thus Theorem~\ref{gmaxg_0} implies that $D_i+\chi_{\{v_j\}}$ is effective. We claim that for all $0 \leq j \leq i \leq n-1$, the divisor $D_i+\chi_{\{v_{2j}\}}$ is equivalent effective. We prove the claim by induction on $j$. If $j=0$, then the assertion of the claim trivially holds. So, assume $j >0$ and let $f=\chi_{\{v_{2j-1},\dots,v_{2i+1}\}}$. A simple computation gives that $D_i+\chi_{\{v_{2j}\}}-Qf=D_{i+1}+\chi_{\{v_{2j-2}\}}$. The induction hypothesis implies that $D_{i+1}+\chi_{\{v_{2j-2}\}}$ is equivalent an effective divisor, so is $D_{i+1}+\chi_{\{v_{2j}\}}$. This shows that $D_i$'s are extreme $v_0$-reduced divisors.

Now assume that $D$ is an extreme $v_0$-reduced divisor. Part (ii) of Lemma~\ref{extreme_reduced_effective lem} implies that $D(v_0)=-1$. If $D(v_{2i+1})=1$ for some $0 \leq i \leq n-1$, then $D$ is not a $v_0$-reduced divisor. The above argument shows that if $D(v_{2i})=2$ or $D(v_{2i})=D(v_{2j})=1$ for some $0 \leq i \ne j \leq n-1$, the divisor $D$ is equivalent to an effective divisor. Obviously $D \neq -\chi_{\{v_0\}}$, and this completes the proof of the claim.

Since each extreme $v_0$-reduced divisor $D_i$, $1\leq i \leq n-1$ has degree zero, $g_{\min}=g_{\max}=g_0$.  Theorem~\ref{g_0g_ming_max_equal} implies that $(G,R)$ has the Riemann-Roch property.
}
\end{proof}
\begin{example}
{
Let $(G,R)$ be an arithmetical graph where $G$ is a cycle $v_1, \dots, v_n$ for $n \geq 3$ and the multiplicity of vertex $v_i$ is $i$ for all $1 \leq i \leq n$. Then $(G,R)$ has Riemann-Roch property.
}
\end{example}
\begin{proof}
{
It is easy to see that $g_0=1$. Now assume $D$ is an extreme $v_1$-reduced divisor. The part (ii) of Lemma~\ref{extreme_reduced_effective lem} implies that $D(v_1)=-1$. If there exists $2 \leq i \leq n$ such that $D(v_i) \geq 1$, then degree of $D$ is at least one. Thus, Theorem~\ref{gmaxg_0} implies that $D$ is equivalent to an effective divisor. This shows that $D=-\chi_{\{v_1\}}$ is the unique extreme $v_1$-reduced divisor and the assertion of the lemma follows Corollary~\ref{unique_extreme}.
}
\end{proof}
The following example introduced in~\cite{Lor09} has the Riemann-Roch property.
\begin{example}
{
Let $(G,R)$ be an arithmetical graph where $G$ is a graph with vertex set $\{v_0,v_1\}$ such that $v_0$ is connected to $v_1$ with $r_0r_1$ edges where $r_0$ and $r_1$ are the multiplicity of the vertex $v_0$ and $v_1$, respectively. Then $(G,R)$ has the
Riemann-Roch property.
}
\end{example}
\begin{proof}
{
The proof follows from Corollary~\ref{unique_extreme}, since there exists a unique
extreme $v_0$-reduced divisor, $D=-\chi_{\{v_0\}}+(r^2_0-1)\chi_{\{v_1\}}$. Hence $g_{\min}=g_{\max}=g_{0}$.
}
\end{proof}
Given any two integers $r_0 > r_1$ we can recursively construct a decreasing sequence $r_i$'s where $r_{i+1}=\delta_ir_i-r_{i-1}$, $r_{i+1} < r_i$ and $\delta_i \in \Bbb N$ for all $i\geq 1$. We call such a sequence the {\it Euclidean sequence generated by $r_0$ and $r_1$}. Note that the Euclidean sequence generated by $r_0$ and $r_1$ is finite and it comes from a simple variation of Euclid's algorithm.

Let $(G,R)$ be an arithmetical graph. We define a {\it Euclidean chain leaving $v_0$ generated by $r_0$ and $r_1$} to be an induced path $C=v_0,v_1 \dots,v_n$ of length $n+1\geq 2$ in $G$ such that $\deg_G(v_{n})=1$ where the corresponding sequence of multiplicities, $r_0,r_1 \dots r_n$ is the Euclidean sequence generated by $r_0$ and $r_1$. Note that $r_n=gcd(r_i,r_{i+1})$ for all $0 \leq i \leq n-1$. If $v_0$, $r_0$ and $r_1$ are clear from the context, we may simply refer to the path as a {\it Euclidean chain}.

Lorenzini~\cite{Lor89} uses a slight variation of the Euclidean chain for building arithmetical graphs. We also use Euclidean chain to construct a arithmetical graph with the Riemann-Roch property.

A {\it Euclidean star generated by} $r_0$ and $r_1$ is an arithmetical graph $(G,R)$ with the {\it center} vertex $v_0$ with multiplicity $r_0$ and $r_0$ identical Euclidean chains leaving $v_0$ generated by $r_0$ and $r_1$. We call the vertex $v_0$ the {\it center vertex}. When $r_0$ and $r_1$ are clear from the context, we will simply say {\it Euclidean star}.

We will show that every Euclidean star generated by $r_0$ and $r_1$ with $gcd(r_0,r_1)=1$, has the Riemann-Roch property.

\begin{definition}
Let $r_0 > r_1$ be two positive integers with $gcd(r_0,r_1)=1$. Assume $r_0, r_1, \dots, r_m$ is the Euclidean sequence generated by $r_0$ and $r_1$.
Given a nonnegative integer $x$, we say $x$ has a good representation with respect to $r_0$ and $r_1$ if there exist $0 \leq t_i \leq \delta_i-1$, for all $1 \leq i \leq m$ such that ${x=\sum_{i=1}^m t_ir_i}$, and there exist no $1 \leq i<j \leq m$ such that $t_i=\delta_i-1$, $t_j=\delta_j-1$ and for all $i<k<j$, $t_k =\delta_k-2$.
\end{definition}
\begin{lemma}
{
\label{unique_representation_lem}
Let $r_0$ and $r_1$ be positive integers with $gcd(r_0,r_1)=1$. Given a nonnegative integer $x$, $x$ has a good representation with with respect to $r_0$ and $r_1$ if and only if $0 \leq x \leq r_0-1$. Moreover, if $0 \leq x \leq r_0-1$ such a representation is unique.
}
\end{lemma}
\begin{proof}
{
Assume $r_0, r_1, \dots, r_m$ is the Euclidean sequence generated by $r_0$ and $r_1$. We prove by induction on $m$. If $m=1$, the assertion of the lemma is obvious. Now assume $m \geq 2$ and $x$ is an arbitrary nonnegative integer. It is easy to see that $t_1 \leq \lfloor {x \over r_1} \rfloor$. If $t_1 < \lfloor {x \over r_1} \rfloor$, then $x-t_1r_1 \geq r_1$, so by the induction hypothesis $x-t_1r_1$ does not have a good representation with respect to $r_1$ and $r_2$ because $gcd(r_1,r_2)=1$ and the Euclidean sequence obtained from $r_1$ and $r_2$ is $r_1,r_2, \dots, r_m$.

Hence, we may assume $t_1 = \lfloor {x \over r_1} \rfloor$, so by induction hypothesis $x-t_1r_1$ has a good representation with respect to $r_1$ and $r_2$. If $t_1 \leq \delta_1-2$, then the good representation of $x-t_1r_1$ with respect to $r_1$ and $r_2$ extends to a good representation of $x$ with respect to $r_0$ and $r_1$.

If $t_1  = \delta_1-1$, then $x-(\delta_1-1)r_1=x-r_0-r_2+r_1 < r_1-r_2$, therefore $x-t_1r_1+r_2=\sum_{i=2}^mt_ir_i$ is a unique good representation with respect to $r_1$ and $r_2$. We claim $t_2 \geq 1$. If $t_2=0$ then $x-t_1r_1+r_2$ has a good representation with respect to $r_2$ and $r_3$, therefore by induction $x-t_1r_1+r_2 < r_2$, so $x-t_1r_1<0$, a contradiction.  Therefore $(t_2-1)r_2+\sum_{i=3}^mt_ir_i$ is the unique good representation of $x-t_1r_1$ with respect to $r_1$ and $r_2$. We claim that $t_1r_1+(t_2-1)r_2+\sum_{i=3}^mt_ir_i$ is the unique good representation of $x$ with respect to $r_0$ and $r_1$. Uniqueness has been established so it remains to show that the representation is good. Assume the representation is not good. It follows that there exists $i \geq 3$ such that $t_i=\delta_i-1$ and for all $2 < k < i$, $t_k=\delta_k-2$, and $t_2-1=\delta_2-2$. Therefore, $t_2=\delta_2-1$, which implies $\sum_{i=2}^mt_ir_i$ is not a  good representation of $x-t_1r_1+r_2$ with respect to $r_0$ and $r_1$, a contradiction.

Suppose there exists an integer $x \geq r_0$ such that $x$ has a good representation with respect to $r_0$ and $r_1$, $x=\sum_{i=1}^mt_ir_i$. If $t_1 \leq \delta_1-2$ then $x-t_1r_1 \geq x-(r_0+r_2)+2r_1 \geq r_1$. So by induction hypothesis $x-t_1r_1$ does not have a good representation respect to $r_1$ and $r_2$, a contradiction. Hence $t_1=\delta_1-1$ and $x-t_1r_1 < r_1$. This implies that $x-t_1r_1 \geq x-(r_0+r_2)+r_1 \geq r_1-r_2$. Let $x-t_1r_1=\sum_{i=2}^m t_ir_i$ be the good representation of $x-t_1r_1$ with respect to $r_1$ and $r_2$. By induction hypothesis $x-t_1r_1+r_2 \geq r_1$ does not have a good representation with respect to $r_1$ and $r_2$.  Either there exists $3 \leq j \leq m$ such that   $t_j=\delta_j-1$, $t_2+1=\delta_2-1$ and $t_i=\delta_i-2$ for all $2<i<j$, or $t_2+1=\delta_2$, both of which contradict the fact that $\sum_{i=1}^mt_ir_i$ is a good representation of $x$ with respect to $r_0$ and $r_1$ because $t_1=\delta_1-1$.
}
\end{proof}
\begin{lemma}
{
\label{v_0_reduced_good_rep_lem}
Let $(G,R)$ be a Euclidean star generated by $r_0$ and $r_1$ with center vertex $v_0$. Then the set of all $v_0$-reduced divisors are the set of divisors such that for any Euclidean chain $C=v_0,v_1, \dots,v_m$ leaving $v_0$, $x=\sum_{i=1}^mD(v_i)r_i$ is a good representation with respect to $r_0$ and $r_1$.
}
\end{lemma}
\begin{proof}
{
Let $D$ be a $v_0$-reduced divisor and $C=v_0,v_1, \dots,v_m$ be a Euclidean chain leaving $v_0$. It is clear that if $x=\sum_{i=1}^mD(v_i)r_i$ is not a good representation with respect to $r_0$ and $r_1$ then $D$ is not a $v_0$-reduced divisor.

Conversely, let $D$ be a divisor such that for every Euclidean chain $C=v_0,v_1, \dots,v_m$ leaving $v_0$, $x=\sum_{i=1}^mD(v_i)r_i$ is a good representation with respect to $r_0$ and $r_1$, but $D$ is not a $v_0$-reduced divisor.
Let $f \geq {\zerov}$ be a firing strategy such that $f(v_0) = 0$ and $D'=D-Qf$ is a $v_0$-reduced divisor. Note that the existence of $f$ is guaranteed by Corollary~\ref{positive_divisor_reduction_cor}. Let $C=v_0,v_1, \dots,v_m$ be a Euclidean chain leaving $v_0$. Without loss of generality we may assume $f' \neq {\zerov}$ where $f'$ is the projection of $f$ into the first $m+1$ coordinates. If $f'(v_1) > 0$ then $\sum_{i=1}^mD'(v_i)r_i <0$, therefore there exists $1 \leq i \leq m$ such that $D'(v_i) < 0$, a contradiction. Hence, $\sum_{i=1}^mD'(v_i)r_i=\sum_{i=1}^mD(v_i)r_i$. Since $f' \neq {\zerov}$, by Lemma~\ref{lem:leftkernel} and the uniqueness of the representation of $\sum_{i=1}^mD(v_i)r_i$ implied by  Lemma~\ref{unique_representation_lem}, $\sum_{i=1}^mD'(v_i)r_i$ is not a good representation. Therefore $D'$ is not $v_0$-reduced, a contradiction.
}
\end{proof}
\begin{definition}
Let $(G,R)$ be a Euclidean star generated by $r_0$ and $r_1$ with the center vertex $v_0$. We say a divisor $S$ is a {\it staircase divisor} if there exists a labeling $C_0, \dots, C_{r_0-1}$ of the Euclidean chains leaving $v_0$ where $P_i=v_{0},v_{i,1}, \dots, v_{i,m}$ is the induced path of $C_i$ such that $\sum_{j=1}^m S(v_{i,j})r_j$ is the good representation of $i$, for all $0 \leq i \leq r_0-1$, and $S(v_0)=-1$.
\end{definition}
\begin{lemma}
{
\label{extreme_divisor_star_lem}
Let $(G,R)$ be a Euclidean star generated by $r_0$ and $r_1$ with the center vertex $v_0$. A divisor $D$ is an extreme $v_0$-reduced divisor if and only if $D$ is a staircase divisor.
}
\end{lemma}
\begin{proof}
{
Let $S$ be a staircase divisor and $C_0, \dots, C_{r_0-1}$ be a labeling of the Euclidean chains leaving $v_0$ where $v_{0},v_{i,1}, \dots, v_{i,m}$ are the vertices of $C_i$. We claim that $S$ is not equivalent to an effective divisor. For proving the claim, it is enough to show that all $v_0$-reduced divisors equivalent to $S$ are staircase divisors. Let $1 \leq k \leq r_0$ and $f_k$ be the firing strategy guaranteed by Corollary~\ref{unique_firing_reduced_cor}, such that $f_k(v_0)=k$ and $S_k=S-Qf_k$ is a $v_0$-reduced divisor. Note that since $S$ is a $v_0$-reduced divisor, by Lemma~\ref{v_0_reduced_good_rep_lem}, the divisor $S$ is $v_0$-reduced. So, as an application of part (ii) of Theorem~\ref{dhar_alg_thm}, we may assume $f_k \geq \zerov$. It is clear from the proof of Lemma~\ref{v_0_reduced_good_rep_lem}, $\sum_{j=1}^mS_k(v_{i,j})r_j$ is a good representation of  $i+kr_1 \hbox{ mod } r_0$ for all $0 \leq i \leq r_0-1$.  Note that $S_k$ is a staircase divisor and $s_k(v_0)=-1$. So (i) of Lemma~\ref{extreme_reduced_effective lem} implies that $S_k$ is not equivalent to an effective divisor.

Now, we prove that for any $v_0$-reduced divisor $D$ not equivalent to an effective, there exists a staircase divisor $S$ such that and $D' \sim D$ such that $D' \leq S$. Let $C_0, \dots, C_{r_0-1}$ be a labeling of the Euclidean chains leaving $v_0$ where $v_{0},v_{i,1}, \dots, v_{i,m}$ are the vertices of $C_i$ such that $\sum_{j=1}^mD(v_{i,j})r_j \leq \sum_{j=1}^mD(v_{i+1,j})r_j$ for all $0 \leq i \leq r_0-2$. Let $S$ be the staircase divisor defined by the same labeling of the Euclidean chains leaving $v_0$. If for all $0 \leq i \leq r_0-1$, $\sum_{j=1}^mD(v_{i,j})r_j \leq i$ then $D \leq S$, so we may assume that there exists $0 \leq i \leq r_0-1$ such that $\sum_{j=1}^mD(v_{i,j})r_j > i$. Let $k$ be such that $kr_1 \equiv r_0-i-1 \hbox{ (mod) } r_0$. By Corollary~\ref{unique_firing_reduced_cor} there exist firing strategies $f_D$ and $f_S$ such that $f_D(v_0)=f_S(v_0)=k$ and the divisors $D_k=D-Qf_D$ and $S_k=S-Qf_S$ are $v_0$-reduced. We claim that $D_k$ is effective, in particular $D_k(v_0)=0$. We have $f_D(v_{\l,1})=f_S(v_{\l,1})=\lfloor {kr_1 \over r_0}\rfloor$ for all $0 \leq \l \leq i-1$ and $f_D(v_{\l,1})=f_S(v_{\l,1})=\lceil {kr_1 \over r_0}\rceil$ for all $i+1 \leq \l \leq r_0-1$, but $f_D(v_{i,1})=\lceil {kr_1 \over r_0}\rceil$ while $f_S(v_{i,1})=\lfloor {kr_1 \over r_0}\rfloor$. This proves the claim and completes the proof of the lemma.
}
\end{proof}
\begin{theorem}
{
\label{euc_star_thm}
Let $(G,R)$ be a Euclidean star then $(G,R)$ has the Riemann-Roch property.
}
\end{theorem}
\begin{proof}
{
By Lemma~\ref{extreme_divisor_star_lem}, we know that the set of staircase divisors is the set of extreme $v_0$-reduced divisors, hence $$g_{\min}-1=g_{\max}-1=(\sum_{i=0}^{r_0-1}i)-r_0=r_0(r_0-3)/2.$$
Let $V(\vec{G})=\{v_0, \dots, v_n\}.$ Using the formula $$g_0-1=\sum_{i=0}^n r_i(\deg(v_i)-2)/2=r_0(r_0-3)/2=\left({\begin{matrix} r_0-1 \cr 2\end{matrix}}\right)-1.$$ Now the assertion of the theorem follows from Theorem~\ref{g_0g_ming_max_equal}.
}
\end{proof}
\subsection{Arithmetical Graphs without the Riemann-Roch Property}
It follows from Theorem~\ref{RR_formula_equiv_U_RI_thm} that an arithmetical graph $(G,R)$ fails to have the Riemann-Roch property if $(G,R)$ is not uniform or is not reflection invariant. The following examples show that all of these three possibilities can happen.
\begin{example}
\label{withot_RR_NU_NR_exa}
{
Let $(G,R)$ be an arithmetical graph, where $G$ is the graph obtained by adding two edges connecting $v_0$ to $v_{3}$ to the $6$-cycle $v_0, \dots, v_{5}$, and the multiplicity of the vertex $v_i$ is $1$ if $i \in \{0,2,4\}$ and is $2$ otherwise. Then $(G,R)$ is neither uniform nor reflection invariant.
}
\end{example}
\begin{proof}
{
Let $\nu_1=-\chi_{\{v_0\}}+\chi_{\{v_{2},v_{3},v_{4}\}}$, $\nu_2=-\chi_{\{v_0\}}+\chi_{\{v_{2}\}}+2\chi_{\{v_{4}\}}$ and $\nu_3=-\chi_{\{v_0\}}+2\chi_{\{v_{2}\}}+\chi_{\{v_{4}\}}$. We claim that $\mathcal{E}=\{\nu_1,\nu_2,\nu_3\}$ is the set of
extreme $v_0$-reduced divisors of $(G,R)$. Note that $\deg_R(\nu_1)=3$ and $\deg_R(\nu_2)=\deg_R(\nu_3)=2$.
For proving the claim we start by showing that $\nu_1$ is
$v_0$-reduced. Let $f$ be a valid firing strategy with respect to $v_0$ such that $(D_1-Qf)(v_i) \geq 0$, for all $1 \leq i \leq 5$. If $f(v_{2})=1$, since $(D_1-Qf)(v_{2}) \geq 0$, we have $f(v_1)+f(v_{3}) \geq 3$. If $f(v_1) =2$, since $(D_1-Qf)(v_{1}) \geq 0$ we must have $f(v_0) \geq 1$, a contradiction. So $f(v_{3})=2$ and this implies that in order to have $(D_1-Qf)(v_{3}) \geq 0$ we must have $f(v_{4}) = 3$, a contradiction. This shows that $f(v_1)=0$, and by symmetry $f(v_5)=f(v_4)=0$, which shows that $f(v_3)=0$. This shows that $f={\zerov}$, which contradicts the fact that $f$ is valid strategy with respect to $v_0$. Hence, $\nu_1$ is $v_0$-reduced, as desired. By applying a similar argument, we can see that $\nu_2$ and $\nu_3$ are $v_0$-reduced divisors. Note that since $r_0=1$, by Lemma~\ref{extreme_reduced_effective lem}(i), the $v_0$-reduced divisors $\nu_1,\nu_2,\nu_3$ are not effective and they are pairwise inequivalent.

It is easy to compute that $\deg_R(\nu_1)=3=g_0-1$, so Theorem~\ref{gmaxg_0} implies that  $\nu_1$ is extreme. Hence, by symmetry, we only need to prove that $\nu_2$ is extreme. For proving this fact it is enough to show that $D=\nu_2+\chi_{\{v_i\}}$ is equivalent to an effective divisor for all $0 \leq i \leq 5$. If $i \not \in \{0,2,4\}$, then degree of $D$ is $4 = g_0$, so Theorem~\ref{gmaxg_0} implies that $D$ is equivalent to an effective divisor. If $i=0$, then $D$ is trivially effective. If $i=2$, then we have a firing strategy $f={\onev}-\chi_{\{v_0\}}$ such that $D-Qf=3\chi_{\{v_0\}} \geq {\zerov}$. Also if $i=4$, then we have $f=\chi_{\{v_4,v_5\}}$ such that $D-Qf=\chi_{\{v_2,v_3\}} \geq {\zerov}$. This completes the proof of the fact that $\nu_1, \nu_2, \nu_3$ are extreme $v_0$-reduced divisors.

Suppose $\nu$ is an extreme $v_0$-reduced divisor. It is easy to see that $\nu(v_2) \leq 2$ (by symmetry $\nu(v_4) \leq 2$), since otherwise $\nu-Qf \geq 0$, where $f=\chi_{\{v_1,v_2\}}$. Note that $\nu(v_1)=\nu(v_5)=0$ and $\nu(v_3) \leq 1$. It follows that $\mathcal{E}$ is the set of $v_0$-reduced divisors and this completes the proof of the claim.
This demonstrates that $(G,R)$ is not uniform.

Now, we are going to show that $(G,R)$ is not reflection invariant.
Let $\L$ be the lattice spanned by Laplacian of $(G,R)$. By applying Lemma~\ref{reduce_exist_lemma} and (ii) of Lemma~\ref{extreme_reduced_effective lem}, we conclude that $Ext(\Sigma(\L))=\{\nu+\l: \ell \in \L, \nu \in \mathcal{E}\}$. Corollary~\ref{extreme_L_ciritical_cor} implies $Crit(\L)=\mathcal{P}+\L$, where $\mathcal{P}=\{\pi(\nu+{\onev}): \nu \in \mathcal{E}\}$. Let $p_i=\pi(\nu_i+\onev)=(\nu_i+\onev)-\left({(\nu_i+\onev) \cdot R \over R \cdot R}\right)R$. An easy computation shows that $p_1={1 \over 5}(-4,-3,6,2,6,-3), p_2={1 \over 15}(-11,-7,19,-7,34,-7)$ and $p_3={1 \over 15}(-11,-7,34,-7,19,-7)$. For seeking a contradiction, assume there exists $v \in \R^{6}$ such that $-Crit(\L)=Crit(\L)+v$. Either there exist $\l,\l',\l'' \in \L$ such that $-p_1=p_1+\l+v$, $-p_2=p_2+\l'+v$ and $-p_3=p_3+\l''+v$, in this case $2(p_i-p_j) \in \L$ for all $1 \leq i \neq j \leq 3$. Or, there exist $\l,\l' \in \L$ and $\{i,j,k\}=\{1,2,3\}$ such that $-p_i=p_j+\l+v$, and $-p_k=p_k+\l'+v$, in this case $-p_j=p_i+\l+v$ and we must have $-2p_k+p_i+p_j \in \L$. Note that $\L \subseteq \Z^6$, so an easy computation shows that none of the above cases happen. This proves that $(G,R)$ is not reflection invariant.
}
\end{proof}
\begin{example}
{
\label{exa:uni_not_ref}
Let $(G,R)$ be an arithmetical graph, where $G$ is a graph obtained from $K_4$ where $V(K_4)=\{v_0,v_1,v_2,v_3\}$, by subdividing the edge $v_2v_3$ twice. The multiplicity of the vertices $v_0$ and $v_1$ are $2$ and $4$ respectively, and the multiplicity of the other vertices are $3$. Then $(G,R)$ is uniform but not reflection invariant.
}
\end{example}
\begin{proof}
{
Let $P=v_2v_4v_5v_3$ be the induced path connecting $v_2$ to $v_3$, i.e., the path obtained by subdividing the edge $v_2v_3$ in the graph $K_4$.

Let $\nu_1=-\chi_{\{v_0\}}+\chi_{\{v_2,v_4\}}$, $\nu_2=-\chi_{\{v_0\}}+2\chi_{\{v_2\}}$ and $\nu_3=-\chi_{\{v_0\}}+2\chi_{\{v_3\}}$.
We claim that $\mathcal{E}=\{\nu_1,\nu_2,\nu_3\}$ is the set of
extreme $v_0$-reduced divisors of $(G,R)$. By running the Generalized Dhar's Algorithm on each $\nu_i$, $1 \leq i \leq 3$, it is not hard to see    that $\nu_1 \sim -\chi_{\{v_0\}}+\chi_{\{v_3,v_5\}}$, $\nu_2 \sim -\chi_{\{v_0\}}+\chi_{\{v_3,v_4\}}$ and $\nu_3 \sim -\chi_{\{v_0\}}+\chi_{\{v_2,v_5\}}$.

We will leave the details of the fact that $\nu_i$, $1 \leq i \leq 3$ is $v_0$-reduced to the reader. (It follows from Lemma~\ref{dhar_alg_thm}, or case analysis similar to that one used in the proof of the Example~\ref{withot_RR_NU_NR_exa}.) It is easy to compute that $g_0=7$, and for all $\nu \in \mathcal{E}$ and $0 \leq i \leq 5$, $\deg_R(\nu+\chi_{\{v_i\}}) \geq 7$. Now, Theorem~\ref{gmaxg_0} implies that $\nu+\chi_{\{v_i\}}$ is equivalent to an effective divisor. This shows that $\nu_i$, $1 \leq i \leq 3$ is extreme $v_0$-reduced.

To finish the proof of the claim, it is enough to show that if $\nu$ is extreme $v_0$-reduced divisor then $\nu \in \mathcal{E}$. Note that $\nu(v_1)=0$ since otherwise $\nu-Qf \geq 0$ where $f=\chi_{\{v_0\}}+3\chi_{\{v_1\}}+2\chi_{\{v_2,v_3,v_4,v_5\}}$. Also, note that if $\nu(v_2) \geq 1$ and $\nu(v_3) \geq 1$, then $\nu-Qf \geq \chi_{\{v_1\}}$ where $f=\chi_{\{v_0, \dots, v_5\}}$. This shows that there exists $1 \leq i \leq 3$ such that $\nu=\nu_i$ or $\nu \sim \nu_i$.

The uniformity of $(G,R)$ immediately follows from the fact that for all $\nu \in \mathcal{E}$, $\deg_R(\nu)=4$.

For proving the fact that $(G,R)$ is not reflection invariant, we apply a similar argument we used in the proof of Example~\ref{withot_RR_NU_NR_exa}. Let $\mathcal{P}=\{p_1,p_2,p_3\}$ be the same set as defined in Example~\ref{withot_RR_NU_NR_exa}. An easy computation shows that $p_1={1 \over 3}(-2,-1,4,-1,4,-1), p_2={1 \over 3}(-2,-1,7,-1,1,-1)$ and $p_3={1 \over 5}(-4,-3,1,7,1,-3)$. For seeking a contradiction, assume there exists $v \in \R^{6}$ such that $-Crit(\L)=Crit(\L)+v$. Either there exist $\l,\l',\l'' \in \L$ such that $-p_1=p_1+\l+v$, $-p_2=p_2+\l'+v$ and $-p_3=p_3+\l''+v$, in this case $2(p_i-p_j) \in \L$ for all $1 \leq i \neq j \leq 3$. Or, there exist $\l,\l' \in \L$ and $\{i,j,k\}=\{1,2,3\}$ such that $-p_i=p_j+\l+v$, and $-p_k=p_k+\l'+v$, in this case $-p_j=p_i+\l+v$ and we must have $-2p_k+p_i+p_j \in \L$. Note that $\L \subseteq \Z^6$, so an easy computation shows that none of the above cases occur. This proves that $(G,R)$ is not reflection invariant.
}
\end{proof}
\begin{example}
{
Suppose $R=(r_0,r_1,r_2)=(1,2,3)$. Let $(G,R)$ be an arithmetical graph where $G$ is a graph with vertex set $\{v_0,v_1,v_2\}$ such that the multiplicity of $v_i$ is $r_i$ and $v_i$ is connected to $v_j$ with $r_ir_j$ edges for all $0 \leq i \neq j \leq 2$. Then $(G,R)$ is not uniform but it is reflection invariant.
}
\end{example}
\begin{proof}
{
We claim that $\nu_1=-\chi_{\{v_0\}}+3\chi_{\{v_1\}}+2\chi_{\{v_2\}}$ and $\nu_2=-\chi_{\{v_0\}}+\chi_{\{v_1\}}+3\chi_{\{v_2\}}$ are the only extreme $v_0$-reduced divisors.
Suppose $\nu$ is an extreme $v_0$-reduced divisor. Lemma~\ref{extreme_reduced_effective lem} (ii) implies that $\nu(v_0)=-1$.
It is not hard to see that $\nu(v_1) \leq 3$ and $\nu(v_2) \leq 3$, otherwise $\nu-Qf$ is effective where $f=\chi_{\{v_1,v_2\}}$ and $f=\chi_{\{v_1\}}+2\chi_{\{v_2\}}$ respectively.
Moreover, if $D=-\chi_{\{v_0\}}+2\chi_{\{v_1\}}+3\chi_{\{v_2\}}$, then $D-Qf$ is effective where $f'=2\chi_{\{v_1\}}+3\chi_{\{v_2\}}$.
Therefore the only possible extreme divisors are $\nu_1$ and $\nu_2$.
By running the generalized Dhar's algorithm on $\nu_1$ and $\nu_2$, and applying Lemma~\ref{dhar_alg_thm}, one can check that $\nu_1$ are $\nu_2$ are $v_0$-reduced and therefore they are not equivalent to effective divisors. Note that the above computation shows that we already checked some of the different possible firing strategies in a run of the generalized Dhar's Algorithm on $\nu_1$ and $\nu_2$.

So, we claim that if an arithmetical graph $(G,R)$ has only two $v_0$-reduced divisors then $(G,R)$ is reflection invariant. Let $\L$ be the lattice spanned by Laplacian of $(G,R)$ and $\mathcal{E}$ be the set of extreme divisors of $\L$. By applying Lemma~\ref{reduce_exist_lemma} and (ii) of Lemma~\ref{extreme_reduced_effective lem}, we conclude that $Ext(\Sigma(\L))=\{\nu+\l: \ell \in \L, \nu \in \mathcal{E}\}$. Corollary~\ref{extreme_L_ciritical_cor} implies $Crit(\L)=\mathcal{P}+\L$ where $\mathcal{P}=\{\pi(\nu+{\onev}): \nu \in \mathcal{E}\}$. Let $\nu_1$ and $\nu_2$ be the only extreme $v_0$-divisors of $(G,R)$ and $p_1=\pi(\nu_1+{\onev})$ and $p_2=\pi(\nu_2+{\onev})$.  For proving the claim its enough to show that $-Crit(\L)=Crit(\L)+v$ where $v=-p_1-p_2$. Assume $p \in Crit(\L)$, therefore there exists $1 \leq i \leq 2$ and $\ell \in \L$ such that $p=p_i+\ell$. Now, it is easy to see that $p_i+\ell+v=-p_j+\ell=-(p_j-\ell)$ where $j=-i+3$ and $p_j-\ell \in Crit(\L)$. This completes the proof of the claim.

So by a similar argument mentioned in proof of Example~\ref{exa:uni_not_ref}, $(G,R)$ is reflection invariant. Since $\deg_R(\nu)=11$ and $\deg_R(\nu')=10$, we have $g_{\max}=12$ and $g_{\min}=11$. This shows that $(G,R)$ is not uniform.
}
\end{proof}
\addcontentsline{toc}{section}{Acknowledgments}
\section*{Acknowledgments}
We would like to thank Matthew Baker for introducing the problem to us and for directing us with helpful discussions and suggestion of potential approaches toward solving the problem. We also thank Dino Lorenzini, Farbod Shokrieh and Robin Thomas for valuable conversations and useful comments.
%
\addcontentsline{toc}{section}{References}

\begin{thebibliography}{10}
\bibitem{AM09}
Omid Amini and Madhusudan Manjunath.
\newblock {R}iemann-{R}och for {S}ub-{L}attice of the {R}oot {L}attice $A_n$.
\newblock Preprint avialable at \url{http://arxiv.org/abs/1007.2454}.

\bibitem{BN07}
Matthew Baker and Serguei Norine.
\newblock {R}iemann-{R}och theorem and {A}bel-{J}acobi theory on a finite
  graph.
\newblock {\em Advances in Mathematics}, 215:766--788, 2007.

\bibitem{Big99}
Norman Biggs.
\newblock Chip-firing and the critical group of a graph.
\newblock {\em J. Algebraic Combin.}, 9(1):25--45, 1999.

\bibitem{BL92}
Andres Bj\"orner and L\'aszl\'o Lov\'asz.
\newblock Chip-firing games on directed graphs.
\newblock {\em J. Algebraic Combin.}, 1(4):305--328, 1992.

\bibitem{BLS91}
Andres Bj\"orner, L\'aszl\'o Lov\'asz, and Peter~Williston Shor.
\newblock Chip-firing games on graphs.
\newblock {\em European J. Combin.}, 12(4):283--291, 1991.

\bibitem{Bol98}
B\'ella Bollob\'as.
\newblock {\em {M}odern {G}raph {T}heory}.
\newblock Springer-Verlag Heidelberg, New York, 1998.

\bibitem{CSM95}
Roger Carter, Graeme Segal, and Ian Macdonald.
\newblock {\em Lectures on Lie Groups and Lie Algebras}.
\newblock Cambridge Univ. Press, Cambridge, 1995.

\bibitem{CL96}
Fan~RK Chung and Robert Langlands.
\newblock A combinatorial laplacian with vertex weights.
\newblock {\em Journal of Combinatorial Theory Series A}, 75(2):316--327, 1996.

\bibitem{CP05}
Denis Chebikin and Pavlo Pylyavskyy.
\newblock A family of bijections between G-parking functions and spanning trees.
\newblock {\em Journal of Combinatorial Theory Series A}, 110(1), 2005.

\bibitem{Dhar90}
Deepak Dhar.
\newblock Self-organized critical state of sandpile automaton models.
\newblock {\em Phys. Rev. Lett.}, 64(14):1613--1616, 1990.

\bibitem{Die05}
Reinhard Diestel.
\newblock {\em Graph Theory}.
\newblock Springer-Verlag Heidelberg, New York, 2005.

\bibitem{HLMPPW08}
Alexander E.~Holroyd, Lionel Levine, Karola Mészáros, Yuyal Peres, James
  Propp, and David B.~Wilson.
\newblock {C}hip-{F}iring and {R}otor-{R}outing on {D}irected {G}raphs.
\newblock {\em Progress in Probability}, 60:331--364, 2008.

\bibitem{GK08}
Andreas Gathmann and Michael Kerber.
\newblock A {R}iemann-{R}och theorem in tropical geometry.
\newblock {\em Mathematische Zeitschrift}, 259:217--230, 2008.

\bibitem{GR01}
Chris~D Godsil and Gordon Royle.
\newblock {\em Algebraic graph theory}.
\newblock Springer-Verlag, New York, 2001.

\bibitem{HKN07}
Jan Hladk{\'y}, Daniel Kr{\'a}l', and Serguei Norine.
\newblock Rank of divisors on tropical curves.
\newblock 2007.
\newblock Preprint avialable at \url{http://arxiv.org/abs/0709.4485}.

\bibitem{Lev11}
Lionel Levine.
\newblock Sandpile groups and spanning trees of directed line graphs.
\newblock {\em J. Combin. Theory Ser. A}, 118:350--364, 2011.

\bibitem{Lor09}
Dino Lorenzini.
\newblock {F}robenius number, {R}iemann-{R}och structure and {Z}eta functions.
\newblock Preprint available at
  \url{http://www.math.uga.edu/~lorenz/paper.html}.

\bibitem{Lor89}
Dino Lorenzini.
\newblock Arithmetical graphs.
\newblock {\em Mathematische Annalen}, 285:481--501, 1989.

\bibitem{Luo08}
Ye~Luo.
\newblock Rank-detemining sets of metric graphs.
\newblock 2008.
\newblock Preprint avialable at \url{http://arxiv.org/abs/0906.2807}.

\bibitem{Mat02}
Jiri Matousek.
\newblock {\em Lectures on Discrete Geometry}.
\newblock Springer-Verlag, New York, 2002.

\bibitem{Mer05}
Criel Merino.
\newblock The chip-firing game.
\newblock {\em Discrete Math.}, 302(1-3):188--210, 2005.

\bibitem{Mer97}
Criel Merino~L\'opez.
\newblock Chip-firing and the {T}utte polynomial.
\newblock {\em Annals of Combinatorics}, 1(3):253--259, 1997.

\bibitem{MZ07}
Grigory Mikhalkin and Ilia Zharkov.
\newblock Tropical curves, their {J}acobians and {T}heta functions.
\newblock 2007.
\newblock Preprint avialable at
  \url{http://arxiv4.library.cornell.edu/abs/math/0612267v2}.

\bibitem{Spe93}
Eugene R.~Speer.
\newblock {A}symmetric {A}belian sandpile models.
\newblock {\em J. Stat. Phys.}, 71:61--74, 1993.

\bibitem{Rud76}
Walter Rudin.
\newblock {\em {P}rinciples of {M}athematical {A}nalysis}.
\newblock McGraw-Hill Book Co., New York, third edition, 1976.
\newblock International Series in Pure and Applied Mathematics.

\bibitem{Sch98}
Alexander Schrijver.
\newblock {\em {T}heory of {L}inear and {I}nteger {P}rogramming}.
\newblock John Wiley \& Sons, 1998.

\bibitem{Sie89}
Carl~Ludwig Siegel.
\newblock {\em {L}ectures on the {G}eometry of {N}umbers}.
\newblock Springer, 1989.

\bibitem{Tar88}
G\'abor Tardos.
\newblock Polynomial bound for a chip-firing game on graphs.
\newblock {\em SIAM J. Discrete Math.}, 1(3):397--398, 1988.

\bibitem{Heu88}
Jan van~den Heuvel.
\newblock Algorithmic aspects of a chip-firing game.
\newblock {\em Combin. Probab. Comput.}, 10(6):505--529, 2001.
\end{thebibliography}

\end{document}